\documentclass
{article}
\usepackage[latin1]{inputenc}
\usepackage{indentfirst}

\usepackage{amsmath,amsfonts,amssymb,amsthm,amscd}
\RequirePackage{ifthen}
\RequirePackage{hyperref}
\usepackage{latexsym}
\usepackage{mathrsfs}
\usepackage{algorithm}
\usepackage[noend]{algorithmic}
\usepackage{paralist}
\usepackage{graphicx}
\usepackage{subfig}
\usepackage{booktabs}
\input xy  \xyoption{all}

\usepackage[a4paper,top=1.3cm,bottom=1.3cm,left=1.3cm,right=1.3cm]{geometry}

\theoremstyle{plain}

\theoremstyle
{plain}
\newtheorem{teorema}{Theorem}[section]
\newtheorem{theorem}[teorema]{Theorem}
\newtheorem{proposizione}[teorema]{Proposition}%[section]
\newtheorem{proposition}[teorema]{Proposition}%[section]
\newtheorem{lemma}[teorema]{Lemma}%[section]
\newtheorem{corollario}[teorema]{Corollary}%[section]
\newtheorem{corollary}[teorema]{Corollary}

\newtheorem{fact}[teorema]{Fact}
\theoremstyle{definition}
\newtheorem{definizione}[teorema]{Definition}%[section]
\newtheorem{esempio}[teorema]{Example}%[section]
\newtheorem{example}[teorema]{Example}%[section]
\newtheorem{remark}[teorema]{Remark}%[section]

\newtheorem{question}[teorema]{Question}
\makeindex

\newcommand{\N}{\mathbb{N}}
\newcommand{\Z}{\mathbb{Z}}
\newcommand{\Q}{\mathbb{Q}}
\newcommand{\R}{\mathbb{R}}

\newcommand{\BB}{\mathcal B}
\newcommand{\BBB}{\mathfrak B}

\newcommand{\U}{\mathscr U}

\newcommand{\E}{\mathcal E}

\newcommand{\II}{\mathcal I}

\newcommand{\XX}{\mathcal X}
\newcommand{\YY}{\mathcal Y}

\def\CS{\mathbf{Coarse/\!_\sim}}
\def\PCS{\mathbf{Coarse}}
\def\SPCS{\mathbf{SCoarse}}
\def\QPCS{\mathbf{QCoarse}}
\def\EN{\mathbf{Entou}}

\def\Set{\mathbf{Set}}

\DeclareMathOperator{\Cay}{Cay}

\DeclareMathOperator{\Mor}{Mor}
\DeclareMathOperator{\cf}{cf}

\DeclareMathOperator{\Sym}{Sym}
\DeclareMathOperator{\USym}{USym}
\DeclareMathOperator{\Ufor}{U}
\DeclareMathOperator{\Wfun}{W}
\DeclareMathOperator{\Ffun}{F}
\DeclareMathOperator{\Gfun}{G}
\DeclareMathOperator{\Ifun}{I}
\DeclareMathOperator{\Jfun}{J}

\author{%Dikran Dikranjan
Nicol\`o Zava\footnote{This work was partially supported by the ``National Group of Algebraic and Geometric Structures, and their Applications'' (GNSAGA - INdAM).}\\  \\ {\footnotesize Department of Mathematical, Computer and Physical Sciences, University of Udine}
\\ {\footnotesize Via Delle Scienze 206, 33100 Udine, Italy}\\ {\footnotesize \tt nicolo.zava@gmail.com}
}
\title{Generalisations of coarse spaces
\thanks{MSC: 18B99, 54E15, 54E25, 54E99, 54A99.
	\endgraf
	{Keywords}: coarse space, quasi-coarse space, semi-coarse space, bornologous map, monoid, coarse equivalence, $\Sym$-coarse equivalence.}}
\date{}
\begin{document}

\maketitle

\begin{abstract}
Coarse geometry, the branch of topology that studies the global properties of spaces, was originally developed for metric spaces and then Roe introduced coarse structures (\cite{Roe}) as a large-scale counterpart of uniformities. In the literature, there are very important generalisations of uniform spaces, such as semi-uniform and quasi-uniform spaces. In this paper, we introduce and start to study their large-scale counterparts, which generalise coarse spaces: semi-coarse spaces and quasi-coarse spaces.
\end{abstract}

%\tableofcontents

\section*{Introduction}

Large-scale geometry, or coarse geometry, is the study of the global properties of spaces, ignoring their local, small-scale ones. It was initially developed for metric spaces and it found many important applications in different branches of mathematics (see \cite{NowYu} for an overview). For example, here we mention the applications to geometric group theory, where a finitely generated group has essentially two structures of metric space induced by its word metrics, up to (see \cite{harpe}). Roe introduced coarse spaces (\cite{Roe}) in order to encode the large-scale properties of spaces and to extend that approach outside the realm of metric spaces. His definition is very similar to the one of uniform space. Recall that a \textit{uniform space} is a pair $(X,\mathcal U)$, where $X$ is a set and $\mathcal U$ is a \textit{uniformity} over it, i.e., a family of subsets of $X\times X$, called entourages, that satisfy the following properties:
\begin{compactenum}[(U1)]
	\item $\mathcal U$ is a filter (i.e., a family closed under taking finite intersections and supersets);
	\item for every $U\in\mathcal U$, $\Delta_X=\{(x,x)\mid x\in X\}\subseteq U$;
	\item for every $U\in\mathcal U$, $U^{-1}=\{(x,y)\mid(y,x)\in U\}\in\mathcal U\}$;
	\item for every $U\in\mathcal U$, there exists $V\in\mathcal U$ such that $V\circ V\subseteq U$ (for every pair of entourages $E,F\subseteq X\times X$, denote $E\circ F=\{(x,z)\mid\exists y\in X:\,(x,y)\in E,\,(y,z)\in F\}$).
\end{compactenum}
For instance, if $(X,d)$ is a metric space, then, if we denote by $B_d(x,R)$ the open ball centered in $x\in X$ with radius $R>0$, the family
\begin{equation}\label{eq:intro}
\mathcal U_d=\{V\subseteq U_R\mid R\ge 0\}\text{, where, for every $R>0$, }U_R=\bigcup_{x\in X}\{x\}\times B_d(x,R),
\end{equation}
is a uniformity over $X$. Dydak and Hoffland in \cite{DydHof} and Protasov in \cite{Pro} independently introduced the large-scale counterparts of the approach to uniformities via coverings. Let us also cite balleans, which are structures equivalent to coarse spaces, that Protasov and Banakh defined to generalise metric balls (\cite{ProBan}).

Uniform spaces have been widely studied since their introduction by the work of Weil and Tukey in the first half of the last century, and successfully applied in different areas. We refer to \cite{Isb} for an introduction to the topic. However, in some cases, uniform spaces have too strong axioms and they cannot parametrise important situations. For example a {\em quasi-metric} of a set $X$ is a map $d\colon X\times X\to\R_{\ge 0}$ that satisfies all the axioms of a metric but the symmetry. Quasi-metrics naturally arises in many situations, for instance they were cited already in Hausdorff monograph \cite{Hau}, when he discussed the Hausdorff metric of a metric space. Moreover, if we allow that $d$ may also take the value $\infty$, then preorders can be described as quasi-metrics. We refer to \cite{Wil} for a general introduction to the subject. Quasi-metrics are innerlu non symmetric, so, if we consider the family $\U_d$ as in \eqref{eq:intro}, then (U3) may not be satisfied.

In order to fill the gap, {\em quasi-uniform spaces} were introduced: a {\em quasi-uniform space} is a pair $(X,\mathcal U)$, where $\mathcal U$ is a {\em quasi-uniformity} over the set $X$, i.e. a family of entourages that satisfies (U1), (U2) and (U4). There is a wide literature investigating those structures and also important applications to computer science were discovered (see the monograph \cite{FleLin} and the survey \cite{Kun} for a wide-range introduction and a broad bibliography). Similarly, a {\em semi-uniform space} is a pair $(X,\mathcal U)$, where $\mathcal U$ is a {\em semi-uniformity} over the set $X$, i.e., a family of entourages that satisfies (U1)--(U3) (see, for example, \cite{Cec}).

The aim of this paper is to introduce large-scale counterparts of quasi-uniform spaces and semi-uniform spaces, respectively, in order to generalise coarse spaces. In particular, we define quasi-coarse spaces and semi-coarse spaces (Definition \ref{def:main}). Moreover, in order to provide a more comprehensive introduction to these new objects, we consider also entourage spaces, which are structures that generalise both quasi-coarse spaces and semi-coarse spaces. First of all, scratching the surface of this topic, we focus on adapting basic notions of coarse geometry (e.g., morphisms, as bornologous maps, connectedness, boundedness) to this more general setting. Moreover, we present a different characterisation of those structures by using ball structures (\cite{ProBan}). 

We motivate our interest in quasi-coarse spaces and semi-coarse spaces by providing a wide list of examples in which those structures naturally appear. Most of them are extensions of some classical examples of coarse spaces. For instance, we widely discuss metric entourage structures (e.g., the structures induced by quasi-metrics), relation entourage structures (for example the ones induced by a preorder on a set), graphic quasi-coarse structures on directed graphs, hyperstructures, finitely generated monoids endowed with word quasi-metrics, and entourage structures on unitary magmas, monoids, and loops. In particular, we prove that every finitely generated monoid can be endowed with precisely just two word quasi-metrics up to asymorphism (Proposition \ref{prop:fin_gen_mon}), which coincide if the monoid is abelian. This result is a generalisation of the classical situation with finitely generated groups endowed with word metrics.

In studying these structures, it is useful to pursue a categorical approach to the subject. We then introduce the categories $\EN$, $\QPCS$, and $\SPCS$ of entourage spaces, quasi-coarse spaces, and semi-coarse spaces, respectively, and bornologous maps between them. The category $\PCS$, of coarse spaces and bornologous maps between them, which was previously studied in \cite{DikZa}, is a full subcategory of all those categories. In the same paper, it was shown that $\PCS$ is topological and we extend this result to $\EN$, $\QPCS$, and $\SPCS$ (Theorem \ref{theo:top_cat}).

The existence of several functors between those four categories turned out to be very useful in a twofold way. First of all, we use them to prove that $\QPCS$ is a reflective subcategory in $\EN$ (but it is not co-reflective), $\SPCS$ is a reflective and co-reflective subcategory in $\EN$, $\PCS$ is a reflective subcategory in $\SPCS$ (but it is not co-reflective), and $\PCS$ is a reflective and co-reflective subcategory in $\QPCS$ (see Theorem \ref{theo:refl_corefl} and \S\ref{sub:cat_constr}). This result helps us in defining some basic categorical constructions, such as products, coproducts and quotients. Furthermore, those functors are a fundamental tool to transport the important notion of closeness of morphisms (and thus coarse equivalences, see \cite{Roe}) from coarse spaces, to the other, weaker structures. In particular, we focus on the notion of $\Sym$-coarse equivalence between quasi-coarse spaces and we provide a characterisation of that property that is similar to the classical one for coarse equivalences between coarse spaces (see Theorem \ref{prop:sym-quasi-coarse-equivalence}).

Finally, we use $\Sym$-coarse equivalences to give important characterisations of some classes of quasi-coarse spaces: metric entourage spaces induced by extended-quasi-metrics and graphic quasi-coarse spaces, giving an answer to a problem posed by Protasov and Banakh (\cite[Problem 9.4]{ProBan}).

The paper is organised as follows. In Section \ref{sec:coarse_spaces} we recall some background in coarse geometry, while we introduce some new notions. For example, here the definitions of quasi-coarse space, semi-coarse space, entourage space are given. Moreover, one of the most important example of those structures, metric entourage structures, is provided and the entourage structures or semi-coarse structures induced by extended semi-positive-definite maps and by extended semi-pseudometric, respectively, are characterised. Moreover, \S\ref{sub:mor} is devoted to discuss morphisms between those spaces and in \S\ref{sub:ball_entou} an equivalent description of those structures by using ball structures is presented. Providing examples of entourage structures is the focus of Section \ref{sec:ex}. We widely discuss relation entourage structures (\S\ref{sub:relation_entou}), graphic quasi-coarse structures (\S\ref{sub:oriented_graph}), hyperstructures (\S\ref{sub:hyper}), finitely generated monoids (\S\ref{sub:fin_gen_mon}) and other entourage structures on algebraic structures, such as unitary magmas, monoids, and loops (\S\ref{sub:algebraic}). In Section \ref{sec:category} the categorical approach is presented. We define the four categories $\EN$, $\SPCS$, $\QPCS$, and $\PCS$ and we prove that they are topological. Moreover, we provide functors between them and show that $\QPCS$ is a reflective subcategory in $\EN$, $\SPCS$ is a reflective and co-reflective subcategory in $\EN$, $\PCS$ is a reflective subcategory in $\SPCS$, and $\PCS$ is a reflective and co-reflective subcategory in $\QPCS$ (\S\ref{sub:fun_sym}). We conclude the section by presenting some categorical constructions, such as products, coproducts and quotients (\S\ref{sub:cat_constr}). Finally, in Section \ref{sec:sym-c.e.}, we introduce the notion of $\Ffun$-coarse equivalence, where $\Ffun$ is a functor from a category $\XX$ to $\PCS$. In particular, we focus on the functor $\Sym\colon\QPCS\to\PCS$ and we characterise $\Sym$-coarse equivalences. We conclude the paper by characterising some special classes of quasi-coarse spaces (\S\ref{sub:characterizing}).

%Even though they are a useful tool, they cannot be applied to every situation. For example, a quasi-metric on a set (i.e., it doesn't need to satisfy the symmetry request) doesn't induce a uniformity. Hence, in the literature (see for example \cite{Cec}), some natural generalisations of uniform spaces appeared. In particular, we cite semi-uniform spaces and quasi-uniform spaces. In the first case, the triangular inequality is relaxed, while, in the second one, the symmetry request is not imposed anymore.

%\subsection{Large-scale topologies on certain algebraic structures}

\section{Generalisations of coarse spaces}\label{sec:coarse_spaces}

%\subsection{From entourages spaces to coarse spaces}

Let $X$ be a set. An {\em entourage} is a subset of the product $X\times X$. For every entourage $E$, every point $x\in X$, and every subset $A$ of $X$, denote 
\begin{equation*}
E[x]=\{y\in X\mid(x,y)\in E\},\quad E[A]=\bigcup_{a\in A}E[a].%,\\ E^{-1}=\{(y,x)\mid(x,y)\in E\},\quad\mbox{and}\quad E\circ F=\{(x,z)\mid\exists y\in X:\,(x,y)\in E,\,(y,z)\in F\}.
\end{equation*}

Let $X$ be a set. An {\em ideal on $X$} is a family $\II$ of subsets of $X$ which is closed under taking subsets and finite unions. Moreover, if $\mathcal F$ is a family of subsets of $X$, we denote by $\widehat{\mathcal F}$ its closure under taking subsets, i.e., $\widehat{\mathcal F}=\{A\subseteq F\mid F\in\mathcal F\}$.

\begin{definizione}\label{def:main}
	Let $X$ be a set. A family $\E\subseteq\mathcal P(X\times X)$ is an {\em entourages structure over $X$} if it is an ideal on $X\times X$ that contains the diagonal $\Delta_X$. Moreover, an entourages structure $\E$ over $X$ is
	\begin{compactenum}[$\bullet$]
		\item a {\em semi-coarse structure} if $E^{-1}\in\E$, for every $E\in\E$;
		\item a {\em quasi-coarse structure} if $E\circ F\in\E$, for every $E,F\in\E$;
		\item a {\em coarse structure} if it is both a semi-coarse and a quasi-coarse structure.
	\end{compactenum}
	The pair $(X,\E)$ is an {\em entourages space} (a {\em semi-coarse space}, a {\em quasi-coarse space}, a {\em coarse space}) if $\E$ is an entourages structure (a semi-coarse structure, a quasi-coarse structure, a coarse structure, respectively) over $X$.
\end{definizione}

If $\E$ is an entourage structure on a set $X$, then also $\E^{-1}=\{E^{-1}\mid E\in\E\}$ is an entourage structure. Of course, $\E=\E^{-1}$ if and only if $\E$ is a semi-structure. Moreover, if $\E$ is a quasi-coarse structure, then $\E^{-1}$ is a quasi-coarse structure.

Let $(X,\E)$ be an entourage space and $Y$ be a subset of $X$. Then $Y$ can be endowed with the {\em entourage substructure} $\E|_Y=\{E\cap (Y\times Y)\mid E\in\E\}$, and $(Y,\E|_Y)$ is called an {\em entourage subspace of $(X,\E)$}. If $\E$ is a quasi-coarse structure (semi-coarse structure), then $\E|_Y$ is a quasi-coarse structure (semi-coarse structure, respectively). 

%Let $X$ be a set and $\F$ a family of subsets of $X$. Define the {\em completion of $\F$} to be the family 
%\begin{equation}
%\label{eq:completion}
%\widehat\F=\{K\subseteq X\mid\exists F\in\F:\,K\subseteq F\}.
%\end{equation}
If $X$ is a set, a family $\BB$ of subsets of $X\times X$ such that $\E=\widehat\BB$ is an entourages structure (semi-coarse structure, quasi-coarse structure, coarse structure, respectively) is a {\em base for the entourages structure} ({\em base for the semi-coarse structure}, {\em base for the quasi-coarse structure}, {\em base for the coarse structure}, respectively) $\E$.% More explicitly, $\BB$ is a base for an entourage structure if and only if the followings hold:
%\begin{compactenum}[(i)]
%	\item there exists $E\in\BB$ such that $\Delta_X\subseteq E$,
%	\item for every $E,F\in\BB$, there exists $K\in\BB$ such that $E\cup F\subseteq K$.
%\end{compactenum}
%Moreover, 
%\begin{compactenum}[(a)]
%	\item $\BB$ is a base for a semi-coarse structure if and only if it satisfies (i), (ii) and, for every $E\in\BB$, there exists $F\in\BB$ such that $E^{-1}\subseteq F$;
%	\item $\BB$ is a base for a quasi-coarse structure if and only if it satisfies (i), (ii) and, for every $E,F\in\BB$, there exists $K\in\BB$ such that $E\circ F\subseteq K$;
%	\item $\BB$ is a base for a coarse structure if and only if it is a base of both a semi-coarse structure and a quasi-coarse structure.
%\end{compactenum}

Let us now give the most important example of these structures. 
\begin{example}\label{ex:metric}
Let $X$ be a set and $d\colon X\times X\to[0,\infty)$ be a map such that $d(x,x)=0$, for every $x\in X$. The map $d$ is a {\em semi-positive-definite map}. Moreover $d$ is a
\begin{compactenum}[$\bullet$]
	\item {\em semi-pseudometric} if, for every $x,y\in X$, $d(x,y)=d(y,x)$;
	\item {\em quasi-pseudometric} if, for every $x,y,z\in X$, $d(x,y)\le d(x,z)+d(z,y)$;
	\item {\em pseudometric} if it is both a quasi-pseudometric and a semi-pseudometric.
\end{compactenum}
Semi-pseudometrics, quasi-pseudometrics and metrics are {\em semi-metrics}, {\em quasi-metrics} and {\em metrics}, respectively, if $x=y$ whenever $d(x,y)=0$. If $d\colon X\times X\to[0,+\infty]$, then the function is called {\em extended}. Moreover the pair $(X,d)$ is called a {\em (extended) semi-pseudometric space}, a {\em (extended) quasi-pseudometric space}, a {\em (extended) pseudometric space}, a {\em (extended) semi-metric space}, a {\em (extended) quasi-metric space}, a {\em (extended) metric space}, respectively. 

A leading example of entourage structures is the metric entourage structure. Let $(X,d)$ be a set endowed with an extended semi-positive-definite map $d$. We define the following entourage structure:%Then the {\em metric entourage structure $\E_d$} on $X$ is defined as follows:
$$
\E_d=\widehat{\{E_R\subseteq X\times X\mid R\ge 0\}},\quad\mbox{where, for every $R\ge 0$},\quad E_R=\bigcup_{x\in X}(\{x\}\times B_d(x,R)).
$$
Even though it is not precise, for the sake of simplicity, we call $\E_d$ a {\em metric entourage structure}. If $d$ is an extended semi-pseudometric, then $\E_d$ is a semi-coarse structure, while, if $d$ is an extended quasi-pseudometric, then $\E_d$ is a quasi-coarse structure. There are non-symmetric quasi-metrics that induce coarse structures. For example consider the quasi-metric space $(\Z,d)$, where $d$ is defined as follows: for every two points $m,n\in\N$,
\begin{equation}\label{eq:quasi_d_but_sym}
d(m,n)=\begin{cases}
\begin{aligned}& n-m &\text{if $m\leq n$,}\\
& 2(m-n) &\text{otherwise.}\end{aligned}
\end{cases}
\end{equation}
Although $d$ is not symmetric, $\E_d$ is a coarse structure.
\end{example}

More examples of entourage spaces will be given in \S\ref{sec:ex}.

An important notion in coarse geometry is boundedness. A subset $A$ of a coarse space $(X,\E)$ is called {\em bounded} if it satisfies one of the following equivalent properties:
\begin{compactenum}[(B1)]
	\item there exists $x\in A$ and $E\in\E$ such that $A\subseteq E[x]$;
	\item for every $x\in A$, there exists $E_x\in\E$ such that $A\subseteq E_x[x]$;
	\item there exists $E\in\E$ such that, for every $x\in A$, $A\subseteq E[x]$ (equivalently, $A\times A\in\E$).
\end{compactenum}
However, if $X$ is an entourage space, although the implications (B3)$\to$(B2)$\to$(B1) hold, (B1)--(B3) are not equivalent anymore as Example \ref{ex:bound} shows.

\begin{example}\label{ex:bound}
	(a) Let $X=\{0,1,2\}$ and consider the semi-coarse structure $\E_1=\widehat{(\{(0,1),(0,2),(1,0),(2,0)\}\cup\Delta_X)}$ and the quasi-coarse structure $\E_2=\widehat{(\{(0,1),(0,2)\}\cup\Delta_X)}$. Then the whole space $X$ satisfies (B1) in both $\E_1$ and $\E_2$, but it doesn't satisfy (B2).
	
	(b) Let $X=\N$ and $d$ and $d^\prime$ be a semi-pseudometric and a quasi-pseudometric defined as follows: for every $m,n\in\N$,
	$$
	d(m,n)=\begin{cases}
	\begin{aligned}&0&\text{if $m=n$,}\\
	&\min\{m,n\}&\text{otherwise,}\end{aligned}
	\end{cases}
	\quad\mbox{and}\quad d^\prime(m,n)=\begin{cases}
	\begin{aligned}&0&\text{if $n>m$,}\\
	&m-n&\text{otherwise.}\end{aligned}
	\end{cases}
	$$
	Then $X$ satisfies (B2) in both the semi-coarse structure $\E_d$ and the quasi-coarse structure $\E_{d^\prime}$, but it doesn't satisfy (B3).
\end{example}
%A subset $A$ of an entourage space $(X,\E)$ is said to be {\em bounded} if, for every $x\in A$, there exists $E_x\in\E$ such that $A\subseteq E_x[x]$. However, this notion could be sometimes too weak, since the entourage depends on the point. Then $A$ is said to be {\em strongly bounded} if there exists $E\in\E$ such that $A\subseteq E[x]$, for every $x\in X$. Every strongly bounded subset is trivially bounded. Moreover, Fact \ref{fact:bound} provides two situations where the implication can be reverted.

An entourage space $(X,\E)$ is {\em locally finite} if, for every $E\in\E$ and $x\in X$, $E[x]$ is finite. Moreover, $X$ has {\em bounded geometry} if there exists a map $\varphi\colon\E\to\N$ such that, for every $E\in\E$ and $x\in X$, $\lvert E[x]\rvert\leq\varphi(E)$. 
%\begin{fact}\label{fact:bound}

Let $(X,\E)$ be a locally finite entourage structure. %Suppose that either $X$ is locally finite or $\E$ is a coarse structure. Then a subset $A$ of $X$ is bounded if and only if it is strongly bounded. 
Then a subset $A$ of $X$ satisfies (B2) if and only if it satisfies (B3). In fact, if $X$ is locally finite, then every subset $A$ satisfying (B2) is finite. Hence $E=\bigcup_{x\in A}E_x\in\E$ and this entourage shows that $A$ satisfies (B3).
%\end{fact}
%\begin{proof}
%Let $A\subseteq X$ be a bounded subset and, for every $x\in A$, let $E_x\in\E$ such that $A\subseteq E_x[x]$. If $X$ is locally finite, then $A$ is finite and $E=\bigcup_{x\in A}E_x\in\E$ shows that $A$ is strongly bounded. If $\E$ is a coarse structure, then, if we fix a point $x\in A$, for every $y,z\in A$, $(y,z)=(y,x)\circ(x,z)\in E_x^{-1}\circ E_x=E\in\E$, which implies the thesis.
%\end{proof}
%The following example shows that we cannot have the same conclusion as in Fact \ref{fact:bound} if $\E$ is either a semi-coarse structure or a quasi-coarse structure.

A family $\{A_i\}_{i\in I}$ of subsets of an entourage space $(X,\E)$ is {\em uniformly bounded} if there exists $E\in\E$ such that, for every $i\in I$ and every $x\in A_i$, $A_i\subseteq E[x]$. In particular, every element of a uniformly bounded family satisfies (B3).

Let $(X,\E)$ be an entourage space and let $x$ and $y$ be two points. If $\{(x,y)\}\in\E$, then we write $x\downarrow y$. If there exist $x_0=x,x_1,\dots,x_{n-1},x_n=y\in X$ such that $x_i\downarrow x_{i+1}$ or $x_{i+1}\downarrow x_i$, for every $i=0,\dots,n-1$, then we write $x\leftrightsquigarrow y$.% We write $x\leftrightsquigarrow y$ if there exist $x_0=x,x_1,\dots, x_{n-1},x_n=y\in X$ such that $x_i\searrow x_{i+1}$ or $x_{i+1}\searrow x_i$, for every $i=0,\dots,n-1$. Note that the relation $\leftrightsquigarrow$ between points of $X$ is an equivalence relation.

An entourage space $X$ is {\em connected} if, %Consider now the following properties for a large-scale topological space $(X,\beta)$:
%\begin{compactenum}
for every $x,y\in X$, $x\leftrightsquigarrow y$. 
%\item[(C$_2$)] for every $x,y\in X$, $x\searrow y$ or $y\searrow x$;
%\item[(C$_{3^\prime}$)] for every $x,y\in X$, $x\searrow y$ and $y\searrow x$;
%\item[(C$_{3^{\prime\prime}}$)] for every $x,y\in X$, $x\downarrow y$ or $y\downarrow x$;
Moreover, $X$ is {\em strongly connected} if, for every $x,y\in X$, $x\downarrow y$ and $y\downarrow x$. Equivalently, an entourage space $(X,\E)$ is strongly connected if and only if $\bigcup\E=X\times X$. Note that a coarse $(X,\E)$ space is connected if and only if it is strongly connected. This property was already introduced in \cite{Roe} in the framework of coarse spaces, so no distinction between them was necessary. 

Since the relation $\leftrightsquigarrow$ is an equivalence relation, for every entourage space $X$ we can consider its {\em connected components}, which are its equivalence classes under that relation.

%We call them {\em connectedness axioms} (more of them are enlisted in \cite{Zav:lst}). They play the role of the large-scale counterpart of separation axioms in topology (see also Remark \ref{rem:conn_vs:sep} to find explanations to this parallelism).

If $(X,\E)$ is an entourage space, we say that $(X,\E)$ is {\em uniformly connected} if there exists $E\in\E$ such that, for every $x,y\in X$, there exists $n\in\N$ such that $(x,y)\in(E\cup E^{-1})^n$. In this case, $X$ is {\em uniformly connected with parameter $E$}.

\begin{example}
	One may ask whether there are quasi-coarse spaces that are strongly connected, but they are not semi-coarse spaces. 
	
	Let $(X,d)$ be a psuedo-metric space and let $h\colon X\to\R$ be an arbitrary function. Then the function $d_h\colon X\to\R_{\ge 0}$, defined by the law
	$$	d_h(x,y)=\begin{cases}\begin{aligned}&d(x,y)+h(y)-h(x)&\text{if $h(y)-h(x)\ge 0$,}\\
	&d(x,y) &\text{otherwise,}\end{aligned}\end{cases}
	$$	
	for every $x,y\in X$, defines a quasi-pseudometric space.
	
	Let now $X=\Z$, $d$ be the usual euclidean metric, and $h(x)=x^3$. Then $(\Z,\E_{d_h})$ is a quasi-coarse space, since $d_h$ is a quasi-metric, and it is strongly connected. However, it is not a coarse space. In fact, for every $R\ge 0$ and every $z\in\R$, $d_h(z+R,z)=R$, while $d_h(z,z+R)=R(1+3z^2+3zR+R^2)$, and the latter strongly depends on the point $z$. Hence, even though $\{(z+R,z)\mid z\in\Z\}\subseteq E_R\in\E_d$, there exists no $S\ge 0$ such that $\{(z,z+R)\mid z\in\R\}\subseteq E_S$.
\end{example}

In Example \ref{ex:metric} we introduced metric entourage structures. We now want to characterise those structures.

Let $(X,\E)$ be an entourage structure. Define its {\em cofinality} as follows: $\cf\E=\inf\{\lvert\BB\rvert\mid\widehat{\BB}=\E\}$.
\begin{proposition}\label{prop:metric}
	Let $(X,\E)$ be an entourage space. 
	\begin{compactenum}[(a)]
		\item Then there exists an extended semi-positive-definite map $d$ on $X$ such that $\E=\E_d$ if and only if $\cf\E\leq\omega$. 
		\item Suppose that $\E$ is a semi-coarse structure. Then there exists an extended semi-metric map $d$ on $X$ such that $\E=\E_d$ if and only if $\cf\E\leq\omega$.
	\end{compactenum}
\end{proposition}
\begin{proof}
First of all, the ``only if'' implications in both items (a) and (b) are trivial since $\{E_n\mid n\in\N\}$ is a base of $\E_d$.

(a, $\gets$) Let $\{F_n\mid n\in\N\}$ be a countable base of $\E$, and, without loss of generality, we can ask that $F_0=\Delta_X$ and $F_n\subseteq F_{n+1}$, for every $n\in\N$. Then define a map $d\colon X\times X\to\N$ as follows: for every $x,y\in X$,
\begin{equation}\label{eq:metrizable}
d(x,y)=\begin{cases}\begin{aligned}
&\min\{n\mid y\in F_n[x]\} &\text{if it exists,}\\
&\infty &\text{otherwise.}
\end{aligned}
\end{cases}
\end{equation}
It is easy to check that $d$ satisfies the required properties.

(b, $\gets$) Suppose that $\E$ is a semi-coarse structure with $\cf\E\leq\omega$. Then we can choose a base $\{F_n\mid n\in\N\}$ as in item (a) with the further property that $E_n=E_n^{-1}$, for every $n\in\N$. Then the map $d$ as in \eqref{eq:metrizable} satisfies the desired properties.
\end{proof}
Note that the maps $d$ in Proposition \ref{prop:metric} are not extended if and only if $(X,\E)$ is strongly connected.

The case where the entourage space is a quasi-coarse space (or a coarse space, in particular, which is a classical result) will be discussed in \S\ref{sub:characterizing}.

We can construct the lattice of entourage structures of a set. If $X$ is a set, denote by $\mathfrak E(X)$ the family of all entourages structures on $X$.
The lattice $\mathfrak E(X)$ is ordered by inclusion. More precisely, let $X$ be a set and $\E,\E^\prime\in\mathfrak E(X)$ be two entourages structures. Then we say that $\E$ is {\em finer} that $\E^\prime$ if $\E\subseteq\E^\prime$ (and $\E^\prime$ is {\em coarser} than $\E$).

Moreover, $\mathfrak E(X)$ has a top element $\mathcal M_X=\mathcal P(X\times X)$ (the {\em indiscrete coarse structure}) and a minimum element $\mathcal T_X=\{\Delta_X\}$ (the {\em discrete coarse structure}). %Note that both the top and the bottom element are actually coarse structures. 
Finally, $\mathfrak E(X)$ is a complete lattice. In fact, for every family $\{\E_i\}_{i\in I}$ of entourage structures, $\bigcap_i\E_i$ is an entourage structure and so their meet $\bigwedge_i\E_i$. Moreover, if $\E_i$ is a semi-coarse structure (quasi-coarse structure), for every $i\in I$, then also $\bigcap_i\E_i$ is a semi-coarse structure (quasi-coarse structure, respectively). Hence, the join of a family of entourage structures (semi-coarse structures, quasi-coarse structures, coarse structures) $\{\E_i\}_{i\in I}$ on a set $X$ can be defined as the entourage structure $\bigvee_i\E_i$ (semi-coarse structure, quasi-coarse structure, coarse structure, respectively) {\em generated} by $\bigcup_i\E_i$, i.e., the finest structure that contains $\E_i$, for every $i\in I$.

\subsection{Morphisms between entourage spaces}\label{sub:mor}

Let us now introduce the morphisms between those spaces. Let $f\colon X\to Y$ be a map between sets. Denote by $f\times f\colon X\times X\to Y\times Y$ the map defined by the law $(f\times f)(x,y)=(f(x),f(y))$, for every $(x,y)\in X\times X$. A map $f\colon(X,\E_X)\to(Y,\E_Y)$ between entourage spaces is said to be
\begin{compactenum}[$\bullet$]
\item {\em bornologous} (or {\em coarsely uniform}, {\em coarse}) if $(f\times f)(E)\in\E_Y$, for every $E\in\E_X$;
\item {\em weakly uniformly bounded copreserving} if, for every $E\in\E_Y$, there exists $F\in\E_X$ such that $(f\times f)(F)=E\cap(f(X)\times f(X))$;
\item {\em uniformly bounded copreserving} if, for every $E\in\E_Y$, there exists $F\in\E_X$ such that, for every $x\in X$, $E[f(x)]\cap f(X)\subseteq f(F[x])$;
\item {\em effectively proper} if, for every $E\in\E_Y$, $(f\times f)^{-1}(E)\in\E_X$.
%\item an {\em asymorphism} if it is bijective and both $f$ and $f^{-1}$ are bounded preserving;
%\item an {\em asymorphic embedding} if $f\colon X\to f(X)\subseteq Y$ is an asymorphism.
\end{compactenum}
%The reason why we use the adverb {\em uniformly} in the previous definition will be clarified in \S\ref{sub:lst}.
Note that all the previous properties can be checked just for all the entourages that belong to some bases of the entourage structures.
\begin{proposition}\label{prop:eff_cop_wcop}
Let $f\colon(X,\E_X)\to(Y,\E_Y)$ be a map between entourage spaces. Then:
\begin{compactenum}[(a)]
	\item if $f$ is effectively proper, then $f$ is uniformly bounded copreserving;
	\item if $f$ is uniformly bounded copreserving, then $f$ is uniformly weakly bounded copreserving.
\end{compactenum}
\end{proposition}
\begin{proof}
(a) Suppose that $f$ is effectively proper and let $E\in\E_Y$. Then, for every $x\in X$, $E[f(x)]\cap f(X)\subseteq f((f\times f)^{-1}(E)[x])$. In fact, for every $y\in X$ such that $(f(x),f(y))\in E$, $(x,y)\in(f\times f)^{-1}(E)$ and so $f(y)\in f((f\times f)^{-1}(E)[x])$.
%y\in f^{-1}(E[f(x)])$, $(f(x),f(y))\in E$ and thus $(x,y)\in (f\times f)^{-1}(E)$.

(b) Suppose now that $f$ is uniformly bounded copreserving and let $E\in\E_Y$. Let $F\in\E_X$ be an entourage such that, for every $x\in X$, $E[f(x)]\cap f(X)\subseteq f(F[x])$. We claim that $E\cap(f\times f)(X\times X)\subseteq(f\times f)(F)$. Let $(x,y)\in E\cap (f\times f)(X\times X)$. There exists $z\in f^{-1}(x)$, and so $y\in E[f(z)]\cap f(X)$, which implies that there exists $w\in F[z]\cap f^{-1}(y)$. Finally, note that $(z,w)\in F$ and $(x,y)=(f(z),f(w))\in(f\times f)(F)$.
\end{proof}
If $f$ is injective, then both implications of Proposition \ref{prop:eff_cop_wcop} can be easily reverted. Proposition \ref{prop:unif_bound_cop_eff_proper} gives another condition that implies their reversibility.% of those arrows.

Note that a map $f\colon(X,\E_X)\to Y$ from an entourage space to a set has uniformly bounded fibers if and only if $R_f=\{(x,y)\in X\times X\mid f(x)=f(y)\}\in\E_X$. We call such a map {\em large-scale injective}.

\begin{proposition}\label{prop:unif_bound_cop_eff_proper}
Let $f\colon(X,\E_X)\to(Y,\E_Y)$ be a map between entourage spaces. If $f$ is effectively proper, then $f$ is large-scale injective. Moreover, if $\E_X$ is a quasi-coarse structure, then the following properties are equivalent:
\begin{compactenum}[(a)] 
%\item % weakly uniformly bounded copreserving and $f$ has uniformly bounded fibers.
%\item Suppose that $\E_X$ is a quasi-coarse structure. If $f$ large-scale injective, then the following properties are equivalent:
%\begin{compactenum}[(b$_1$)]
	\item $f$ is large-scale injective and it is weakly uniformly bounded copreserving;
	\item $f$ is large-scale injective and it is uniformly bounded copreserving;
	\item $f$ is effectively proper.
%\end{compactenum}
\end{compactenum}
\end{proposition}
\begin{proof}
The first statement can be easily proved: since $\Delta_Y\in\E_Y$, then $R_f=(f\times f)^{-1}(\Delta_Y)\in\E_X$.% and, moreover, for every $E\in\E_Y$, $(f\times f)((f\times f)^{-1}(E))=E\cap(f(X)\times f(X))$.

In view of Proposition \ref{prop:eff_cop_wcop}, we just need to show the implication (a)$\to$(c). Suppose now that $f$ is weakly uniformly bounded copreserving and $R_f\in\E_X$. Let $E\in\E_Y$ and $(x,y)$ be an arbitrary point in $(f\times f)^{-1}(E)$. Let $F\in\E_X$ such that $(f\times f)(F)=E\cap(f(X)\times f(X))$. Then there exists $(z,w)\in F$ such that $(f(x),f(y))=(f(z),f(w))$ and thus
\begin{equation}
(x,y)=(x,z)\circ(z,w)\circ(w,y)\in R_f\circ F\circ R_f\in\E_X.\tag*{\qedhere}
\end{equation}
\end{proof}

\begin{proposition}
Let $f\colon(X,\E_X)\to(Y,\E_Y)$ be a uniformly bounded copreserving surjective map between entourage spaces. Then $Y$ has bounded geometry ($Y$ is locally finite) whenever $X$ has bounded geometry ($X$ is locally finite, respectively).
\end{proposition}
\begin{proof}
Suppose that $\varphi\colon\E_X\to\N$ is a map that demonstrates that $(X,\E_X)$ has bounded geometry. Let $E\in\E_Y$. Then there exists $F\in\E_X$ such that, for every $x\in X$, $E[f(x)]\subseteq f(F[x])$. Hence, $\lvert E[f(x)]\rvert\leq\lvert F[x]\rvert\leq\varphi(F)$. The other implication can be similarly proved.
\end{proof}

A bijective map $f\colon(X,\E_X)\to(Y,\E_Y)$ between entourage spaces is called an {\em asymorphism} if it satisfies the following, equivalent, properties:
\begin{compactenum}[$\bullet$]
\item $f$ and $f^{-1}$ are bornologous;
\item $f$ is bornologous and weakly uniformly bounded copreserving;
\item $f$ is bornologous and uniformly bounded copreserving;
\item $f$ is bornologous and effectively proper.
\end{compactenum}

%Let $f\colon(X,d_x)\to(Y,d_Y)$ be a map between sets endowed with extended-semi-positive-defined maps. Then $f$ is a {\em large-scale Lipschitz map} if there exists $L>0$ and $C\ge 0$ such that $d_Y(f(x),f(y))\leq Ld_X(x,y)+C$ (here we assume that, if $d_Y(f(x),f(y))=\infty$, then also $d_X(x,y)$), for every $x,y\in X$. If the constants $L$ and $C$ can be taken to be $1$ and $0$, respectively, then $f$ is said to be {\em non-expanding}. If $f\colon(X,d_x)\to(Y,d_Y)$ is a large-scale Lipschitz map between sets endowed with extended-semi-positive-defined maps, then $f\colon(X,\E_{d_X})\to(Y,\E_{d_Y})$ is uniformly bounded preserving.

Let $(X,\E_X)$ and $(Y,\E_Y)$ be two asymorphic entourage spaces. Then $\E_X$ is a semi-coarse structure (quasi-coarse structure) if and only if $\E_Y$ is a semi-coarse structure (quasi-coarse structure, respectively). For the proof of this fact, we address to \cite{ProBan}, where the authors used the equivalent approach through ball structures (see \S\ref{sub:ball_entou} for the introduction of these structures).

Furthermore, if $X$ and $Y$ are two asymorphic entourage spaces, then $X$ is (strongly or uniformly) connected if and only if $Y$ is (respectively strongly or uniformly) connected.

\subsection{Approach via ball structures}\label{sub:ball_entou}

Let $(X,\E)$ be an entourage structure. Then we can associate to it a triple $\BBB_{\E}=(X,P,B_\E)$, where $P=\{E\in\E\mid\Delta_X\subseteq E\}$ and $B_\E(x,E)=E[x]$, for every $x\in X$ and every $E\in P$. It is an example of {\em ball structure}.

\begin{definizione}\label{def:ball_structure} (\cite{ProBan,ProZar})
	A {\em ball structure} is a triple $\BBB=(X,P,B)$ where $X$ and $P$ are sets, $P\neq\emptyset$, and $B\colon X\times P \to \mathcal P(X)$ is a map, such that 
	$x\in B(x,r)$ for every $x\in X$ and every $r\in P$. The set $X$ is called {\em support of the ball structure}, $P$ -- {\em set of radii}, and $B(x,r)$ -- {\em ball of center $x$ and radius $r$}. In case $X=\emptyset$, the map $B$ is the empty map. 
\end{definizione}
%A leading example of ball structure, that actually justifies the notation and the terminology, is the metric ball structure. If $(X,d)$ is a metric space ($d$ could have also weaker properties), then the {\em metric ball structure} $\BBB_d$ is the triple $(X,\R_{\ge 0},B_d)$, where, for every $x\in X$ and every $R\ge 0$, $B_d(x,R)$ is the metric ball centred in $x$ with radius $R$.

The terminology and the intuition come from the metric setting: if $(X,d)$ is a metric space, then $\BBB_d=(X,\R_{\ge 0},B_d)$, where $B_d(x,R)$ is the closed ball centred in $x\in X$ with radius $R\ge 0$, is a ball structure.

For a ball structure  $(X,P,B)$, $x\in X$, $r\in P$ and a subset $A$  of $X$, one puts
$$
B^*(x,r)=\{y\in X\mid x\in B(y,r)\}\quad\quad B(A,r)=\bigcup\{B(x,r)\mid x\in A\}.
$$

A ball structure $\BBB=(X,P,B)$ is said to be:
\begin{compactenum}[$\bullet$]
	\item {\em weakly upper multiplicative} if, for every pair of radii $r,s\in P$ there exists $t\in P$ such that $B(x,r)\cup B(x,s)\subseteq B(x,t)$, for every $x\in X$;
	\item {\em upper multiplicative} if, for every pair of radii $r,s\in P$ there exists $t\in P$ such that $B(B(x,r),s)\subseteq B(x,t)$, for every $x\in X$; 
	\item {\em upper symmetric} if, for every pair of radii $r,s\in P$ there exist $r^\prime,s^\prime\in P$ such that $B^\ast(x,r)\subseteq B(x,r^\prime)$ and $B(x,s)\subseteq B^\ast(x,s^\prime)$, for every $x\in X$.
	%\item {\em symmetric} if, for every radius $\alpha\in P$, $i(\alpha)=\alpha$.
\end{compactenum}
It is trivial that upper multiplicativity implies weak upper multiplicativity since every ball contains its center.

%Of course, since the center of a ball belongs to the ball, an upper multiplicative ball structure is weakly upper multiplicative. Moreover a symmetric ball structure is upper symmetric.%both upper and lower symmetric.

\begin{definizione}
	A ball structure is
	\begin{compactenum}[$\bullet$]
		\item a {\em semi-ballean} if it is weakly upper multiplicative and upper symmetric;
		\item a {\em quasi-ballean} if it is upper multiplicative;
		\item a {\em ballean} (\cite{ProBan}) if it is both a semi-ballean and a quasi-ballean.
	\end{compactenum}
\end{definizione}
For every entourage space $(X,\E)$, $\BBB_{\E}$ is indeed a weakly upper multiplicative ball structure. Moreover, if $\E$ is a semi-coarse structure, then $\BBB_{\E}$ is a semi-ballean, while, if $\E$ is a quasi-coarse structure, then $\BBB_{\E}$ is a quasi-ballean.

We have seen how we construct ball structures from entourage structures. Let us now discuss the opposite construction. Let $\BBB=(X,P,B)$ be a weakly multiplicative ball structure. Then we can define an associated entourage structure $\E_{\BBB}$ of $X$ as follows: for every $r\in P$,
$$
E_r=\bigcup_{x\in X}(\{x\}\times B(x,r)),
$$
and the family $\{E_r\mid r\in P\}$ is a base for the entourage structure $\E_\BBB$. Moreover, 
\begin{compactenum}[$\bullet$]
	\item if $\BBB$ is a semi-ballean, then $\E_\BBB$ is a semi-coarse structure;
	\item if $\BBB$ is a quasi-ballean, then $\E_{\BBB}$ is a quasi-coarse structure;
	\item if $\BBB$ is a ballean, then $\E_{\BBB}$ is a coarse structure.
\end{compactenum}
%It is important to observe that the way we induce entourage structures from weakly multiplicative ball structures is precisely the same construction we use to define metric entourage structures.

Let $\BBB$ and $\BBB^\prime$ be two weakly multiplicative ball structure on the same support $X$. Then we identify those two ball structure, and we write $\BBB=\BBB^\prime$, if $\E_{\BBB}=\E_{\BBB^\prime}$. We soon give a characterization of the equality between ball structures. Hence, for every entourage space $(X,\E)$ and every weakly multiplicative ball structure $\BBB$ on $X$,
$$
\E_{\BBB_\E}=\E\quad\mbox{and}\quad\BBB_{\E_\BBB}=\BBB.
$$
The equivalence between coarse structures and balleans have already been widely discussed (see, for example, \cite{ProZar,DikZa}).

Let $\BBB=(X,P_X,B_X)$ and $\BBB_Y=(Y,P_Y,B_Y)$ be two weakly multiplicative ball structures and $f\colon\BBB_X\to\BBB_Y$ be a map. The map $f$ is {\em bornologous} if the following equivalent properties are fulfilled:
\begin{compactenum}[$\bullet$]
	\item $f\colon(X,\E_{\BBB_X})\to(Y,\E_{\BBB_Y})$ is bornologous;
	\item for every radius $r\in P_X$, there exists $s\in P_Y$ such that $f(B_X(x,r))\subseteq B_Y(f(x),s)$, for every $x\in X$.
\end{compactenum}
Similarly, $f$ is {\em uniformly bounded copreserving} if the following equivalent properties are satisfies:
\begin{compactenum}[$\bullet$]
	\item $f\colon(X,\E_{\BBB_X})\to(Y,\E_{\BBB_Y})$ is uniformly bounded copreserving;
	\item for every $s\in P_Y$, there exists $r\in P_X$ such that $B_Y(f(x),s)\cap f(X)\subseteq f(B_X(x,r))$, for every $x\in X$.
\end{compactenum}
Thanks to this characterisation of being uniformly bounded copreserving, it is clear that this notion generalises the one of $\succ$-mapping (\cite{ProZar}): the map $f$ is a {\em $\succ$-mapping} if, for every $s\in P_Y$, there exists $r\in P_X$ such that $B_Y(f(x),s)\subseteq f(B_X(x,r))$, for every $x\in X$. Of course, a surjective map is uniformly bounded copreserving if and only if it is a $\succ$-mapping. However, the second definition is very restrictive when the map is not surjective. In fact, if a map $f\colon(X,\E_X)\to(Y,\E_Y)$ is a $\succ$-mapping, then, if $(f(x),y)\in E$ for some $E\in\E_Y$, $x\in X$ and $y\in Y$, then $y\in f(X)$.% If, moreover, $\E_Y$ is a semi-coarse structure, then $Y$ splits as {\em coproduct} (see \S\ref{sub:cat_constr}) of $f(X)$ and $Y\setminus f(X)$.

Finally, let us give the promised characterisation of the equality between ball structures on the same support. If $\BBB$ and $\BBB^\prime$ are two ball structure on a set $X$, then $\BBB=\BBB^\prime$ if and only if $id_X\colon\BBB\to\BBB^\prime$ is an \emph{asymorphism}, i.e., bornologous with also its inverse bornologous.

We have briefly recalled how coarse spaces and balleans are equivalent constructions. In the literature, there is a third way to describe coarse spaces by using coverings: the so-called {\em large-scale structures} (\cite{DydHof}, also know as {\em asymptotic proximities} in \cite{Pro}). Those are large-scale counterpart of the classical approach to uniformities via coverings (see \cite{Isb}). Moreover, in \cite{PicPul}, the authors presented a way to use the covering approach to describe quasi-uniformities. Hence the following question naturally arises.
\begin{question}
	Is it possible to give a characterisation of entourage structures, semi-coarse structures or quasi-coarse structures through coverings?
\end{question}

\section{Some examples of entourage spaces}\label{sec:ex}

In this section we enlist some examples of entourage spaces.

%If $(X,\leq)$ is a set endowed with a preorder, the {\em Alexandroff topology} can be defined (\cite{Al}) and it is an important object in topology. Hence, our first example concerns relations on a set.
\subsection{Relation entourage structures}\label{sub:relation_entou}
%\begin{example}
Let $X$ be a set and $\mathscr R$ be a reflexive relation of $X$. Define $\E_{\mathscr R}=\widehat{\{\mathscr R\}}$, which is an entourage structure, called {\em relation entourage structure}. Moreover, $\mathscr R$ is symmetric if and only if $\E_{\mathscr R}$ is a semi-coarse structure, while $\mathscr R$ is transitive if and only if $\E_{\mathscr R}$ is a quasi-coarse structure. % The large-scale topology $\beta_{\mathscr R}=\beta_{\E_\mathscr R}=\{\beta_{\mathscr R}(x)\mid x\in X\}$ induced by $\E_{\mathscr R}$ can be described as follows: for every $x\in X$, $\beta_{\mathscr R}(x)=\widehat{\{\mathscr R[x]\}}$.
Furthermore, note that, $(\E_{\mathscr R})^{-1}=\E_{\mathscr R^{-1}}$, where $\mathscr R^{-1}$ denotes the inverse of $\mathscr R$ as an entourage. Another entourage structure that can be defined from a reflexive relation $\mathscr R$ on a set $X$ is the following: $\E_{\mathscr R}^{fin}=[\mathscr R]^{<\infty}\cup\{\Delta_X\}$.

It is easy to verify the following result.
\begin{proposition}
Let $f\colon(X,\mathscr R_X)\to(Y,\mathscr R_Y)$ be a map between sets endowed with reflexive relations. Then the following properties are equivalent:
\begin{compactenum}[(a)]
\item $f$ {\em preserves the relation} (i.e., for every $x,y\in X$, $f(x)\mathscr R_Yf(y)$ provided that $x\mathscr R_Xy$);
\item $f\colon(X,\E_{\mathscr R_X})\to(Y,\E_{\mathscr R_Y})$ is bornologous;
\item $f\colon(X,(\E_{\mathscr R_X})^{-1})\to(Y,(\E_{\mathscr R_Y})^{-1})$ is bornologous;
\item $f\colon(X,\E_{\mathscr R_X}^{fin})\to(Y,\E_{\mathscr R_Y}^{fin})$ is bornologous;
\item $f\colon(X,(\E_{\mathscr R_X}^{fin})^{-1})\to(Y,(\E_{\mathscr R_Y}^{fin})^{-1})$ is bornologous.
%\item $f\colon(X,\beta_{\mathscr R_X})\to(Y,\beta_{\mathscr R_Y})$ is bounded preserving;
%\item $f\colon(X,\beta^{fin}_{\mathscr R_X})\to(Y,\beta^{fin}_{\mathscr R_Y})$ is bounded preserving.
\end{compactenum}
\end{proposition}

We have discussed how one can construct entourage structures from reflexive relations. Now, we focus on the opposite process. Let $(X,\E)$ be an entourage space. Then we define $\mathscr R_\E=\bigcup\E$, which is a reflexive relation since $\Delta_X\in\E$. Moreover, if $\E$ is a semi-coarse structure, then $\mathscr R_\E$ is symmetric, and, if $\E$ is a quasi-coarse structure, then $\mathscr R_\E$ is transitive.

%\begin{compactenum}[$\bullet$]
Note that, if $\mathscr R$ is a reflexive relation on $X$, then 
$$
\mathscr R=\mathscr R_{\E_\mathscr R}=\mathscr R_{\E^{fin}_\mathscr R}.
$$
%\item If $\beta$ is a large-scale topology of $X$, then
%$$
%\beta^{fin}_{\mathscr R_\beta}\leq\beta\leq\beta_{\mathscr R_\beta}.
%$$
%Moreover, $\beta^{fin}_{\mathscr R_\beta}=\beta$ if and only if $\beta$ is locally finite. Furthermore, $\beta=\beta_{\mathscr R_\beta}$ if and only if, for every $x\in X$, $\bigcup\beta(x)\in\beta(x)$. We call the latter large-scale topology {\em Alexandroff large-scale topology}.
Meanwhile, if $(X,\E)$ is an entourage space, then
\begin{equation}\label{eq:rel}
\E_{\mathscr R_\E}^{fin}\subseteq\E\subseteq\E_{\mathscr R_\E}.
\end{equation}
The inclusions in \eqref{eq:rel} can be strict. Consider, for example, $\R$ endowed with the usual metric $d$. Then $\E_{\mathscr R_{\E_d}}^{fin}\subsetneq\E_d\subsetneq\E_{\mathscr R_{\E_d}}$. Furthermore, note that $\E=\E_{\mathscr R_\E}$ if and only if $\bigcup\E\in\E$ and, thus, every entourage structure $\E$ on a finite set $X$ is a relation entourage structure.% (in other words, there exists an entourage $E\in\E$ such that $\E=\widehat{\{E\}}$). Conditions for which the equality $\E^{fin}_{\mathscr R_\E}=\E$ holds will be discussed later\footnote{I need the notion of coproduct of entourage structures}.

\subsection{Graphic quasi-coarse structures}\label{sub:oriented_graph}
%\begin{esempio}
Let $\Gamma=(V,E)$ be a directed graph. %Then we can associate a large-scale topology to the set of vertexes $V$ as follows: for every $x\in V$, $\beta_{\Gamma}(x)=\{y\in V\mid(x,y)\in E\}$. However, this large-scale topology is not very interesting. Here we present two other quasi-coarse structures that looks much more worth of attention.
%\begin{compactenum}[(a)]
%\item 
Define the {\em path extended quasi-metric over $V$} to be the value:
$$ d(v,w)=\begin{cases}\begin{aligned}&\min\{\lvert\{(x_i,x_{i+1})\}_{i=0}^{n-1}\rvert\mid\forall i=0,\dots,n-1, (x_i,x_{i+1})\in E,\,x_0=v,\,x_n=w\} &\text{if it exists,}\\
&\infty &\text{otherwise.}\end{aligned}\end{cases}
$$
This is an extended quasi-metric and $\E_{d}$ is called {\em graphic quasi-coarse structure}.
	
%\item 
Let The graphic quasi-coarse space can be extended to the points on the graph edges, by identifying every edge with the interval $[0,1]$ endowed to the relation quasi-coarse structure associated to the usual order $\leq$ on $[0,1]$. More precisely, if $\Gamma=(V,E)$ is a directed graph and $(v,w)\in E$, then we identify $0$ and $v$ and $1$ and $w$, respectively. This new quasi-coarse structure is called {\em extended graphic quasi-coarse structure}.
%\item Let $w\colon E\to\R_{\ge 0}$ be a weight function that associates to each edge a positive weight. An extended quasi-metric $\overline d_w\colon V\times V\to E$ can be defined by the following law: for every $x,y\in V$,
%$$
%d_w(x,y)=\begin{cases}\begin{aligned}&\inf\bigg\{\sum_{i=0}^{n-1}w((x_i,x_{i+1}))\mid\forall i=0,\dots,n-1, (x_i,x_{i+1})\in E,\,x_0=v,\,x_n=w\bigg\} &\text{if $w\in\QQ_{\beta_{\Gamma}}^\searrow(v)$,}\\
%&\infty&\text{otherwise.}\end{aligned}\end{cases%}
%$$
%The quasi-coarse space $\E_{d_w}$ generated is the {\em weight graphic quasi-ballean}. %Moreover, the {\em extended weight graphic quasi-coarse space} can be defined by uniformly distributing the weight of a edge
%\end{compactenum}

Let $f\colon\Gamma(V,E)\to\Gamma^\prime(V^\prime,E^\prime)$ be a map between oriented graphs. Then $f$ is said to be a {\em graph homomorphism} if, for every $(x,y)\in E$, either $f(x)=f(y)$ or $(f(x),f(y))\in E^\prime$. If $f\colon\Gamma(V,E)\to\Gamma^\prime(V^\prime,E^\prime)$ is a graph homomorphism, then $f$ sends directed paths into non-longer directed paths. Hence $f\colon(V,d)\to(V^\prime,d)$ is \emph{non-expanding} (i.e., $d(f(x),f(y))\leq d(x,y)$, for every $x,y\in V$), and thus $f\colon(V,\E_d)\to(V^\prime,\E_d)$ is bornologous. %Similarly, if $w$ and $w^\prime$ are weight functions on $\Gamma$ and $\Gamma^\prime$, respectively, then $f\colon(V,\E_{d_w})\to(V,\E_{d_{w^\prime}})$ is uniformly bounded preserving, provided that $f$ satisfies the following further condition: for every $(x,y)\in E$, $w^\prime(f((x,y)))\leq w((x,y))$ (assuming that $w^\prime((f(x),f(y)))=0$ whenever $f(x)=f(y)$).
%\end{esempio}
%An {\em ideal} $\II$ of a set $X$ is a family of subsets of $\II$ which is closed under finite union and under taking subsets. A family $\JJ\subseteq\mathcal P(X)$ is a {\em base of the ideal $\II$}, if $\widehat{\JJ}=\II$. 

\subsection{Entourage hyperstructures}\label{sub:hyper}
Let $(X,\E)$ be an entourage structure. We define the following two entourage structures on $\mathcal P(X)$: 
$$
\mathcal H(\E)=\widehat{\{\mathcal H(E)\mid \Delta_X\subseteq E\in\E\}}\quad\mbox{and}\quad\exp\E=\widehat{\{\exp E\mid\Delta_X\subseteq E\in\E\}}=\mathcal H(\E)\cap\mathcal H(\E)^{-1},
$$
where, for every $E\in\E$,
$$
\mathcal H(E)=\{(A,B)\mid B\subseteq E[A]\}\quad\mbox{and}\quad\exp(E)=\mathcal H(E)\cap\mathcal H(E)^{-1},
$$
named {\em entourage hyperstructure} and {\em semi-coarse hyperstructure}, respectively. The way we obtained semi-coarse hyperstructures from entourage hyperstructures will be generalised in \S\ref{sub:fun_sym}. In \cite{DikProProZav}, the authors defined the notion of {\em hyperballean}. It is an equivalent way to define the semi-coarse hyperspace of a coarse space in terms of balleans. Moreover, hyperballeans already appeared in \cite{ProPro}, although, the authors just considered the subspace of the hyperballean whose support is the family of all non-empty bounded subsets.% Note that the construction of the coarse hyperspace is similar to the one of the uniform hyperspace of a uniform space (see \cite{Isb}).

First of all, note that, if $\E$ is an entourage structure, then both $\mathcal H(\E)$ and $\exp\E$ are entourage structures since $\mathcal H(E)\cap\mathcal H(F)=\mathcal H(E\cup F)$, for every $E,F\in\E$. More precisely, $\exp\E$ is actually a semi-coarse structure. Furthermore, if $\E$ is quasi-coarse structure, then $\mathcal H(\E)$ is a quasi-coarse structure, while $\exp\E$ is a coarse structure. In fact, for every $E,F\in\E$, if $(A,B)\circ(B,C)\in\mathcal H(E)\circ\mathcal H(F)$, then $B\subseteq E[A]$ and $C\subseteq F[B]$, which implies that $C\subseteq F[E[A]]=(F\circ E)[A]$ and so $(A,C)\in\mathcal H(F\circ E)$. Note that $\mathcal H(\E)$ is not a semi-coarse structure, unless the support $X$ of $\E$ is empty: in fact, $(X,\emptyset)\in\mathcal H(\Delta_X)$, although, for every $E\in\E$, $E[\emptyset]=\emptyset$. Moreover, even if we consider the subspace $(\mathcal P(X)\setminus\{\emptyset\},\mathcal H(\E)|_{\mathcal P(X)\setminus\{\emptyset\}})$, it is a semi-coarse structure if and only if $X$ satisfies (B3). In fact, $(X,\{x\})\in\mathcal H(\Delta_X)$, for every $x\in X$.

%There is a natural inclusion of $X$ in its power set $\mathcal P(X)$: we define $i\colon X\to\mathcal P(X)$, where, for every $x\in X$, $i(x)=\{x\}$. If $(X,\E)$ is an entourage structure, then $i\colon(X,\E)\to(i(X),\mathcal H(\E)|_{i(X)})$ is an asymorphism. Moreover, if $(X,\E)$ is a semi-coarse structure, then also $i\colon(X,\E)\to(i(X),\mathcal \exp\E|_{i(X)})$ is an asymorphism. If $\E$ is not a semi-coarse structure on $X$, then $i\colon(X,\E)\to(i(X),\mathcal \exp\E|_{i(X)})$ is not an asymorphism, since $(i(X),\mathcal \exp\E|_{i(X)})$ is a semi-coarse space.

Every map $f\colon X\to Y$ between sets can be extended to a map $\overline f\colon\mathcal P(X)\to\mathcal P(Y)$ such that, for every $A\in\mathcal P(X)$, $\overline f(A)=f(A)\in\mathcal P(Y)$. 
\begin{proposition}
Let $f\colon(X,\E_X)\to(Y,\E_Y)$ be a between entourage spaces. The following properties are equivalent:
\begin{compactenum}[(a)]
\item $f\colon(X,\E_X)\to(Y,\E_Y)$ is bornologous;
\item $\overline f\colon(\mathcal P(X),\mathcal H(\E_X))\to(\mathcal P(Y),\mathcal H(\E_Y))$ is bornologous.
\end{compactenum}
\end{proposition}
\begin{proof}
As for the implication (a)$\to$(b), if $f$ is bornologous, then the inclusion $(\overline f\times\overline f)(\mathcal H(E))\subseteq\mathcal H((f\times f)(E))$, for every $E\in\E_X$, holds, and the thesis follows. Conversely, (b)$\to$(a) is a consequence of the fact that, for every entourage space $(Z,\E_Z)$, if $E\in\E_Z$ and $x,y\in Z$, then $(x,y)\in E$ if and only if $(\{x\},\{y\})\in\mathcal H(E)$.
\end{proof}
%Moreover, they implies
%\begin{compactenum}[(c)]
%\item $f\colon(\mathcal P(X),\exp\E_X)\to(\mathcal P(X),\exp\E_Y)$ is bornologous.
%\end{compactenum}
%If, furthermore, $\E_X$ and $\E_Y$ are semi-coarse structures, then item (c) is equivalent to (a) and (b) since $\mathcal H(\E_X)=\exp(\E_X)$ and $\mathcal H(\E_Y)=\exp(\E_Y)$.

%The request of symmetry may seem too strong and one can come out with a maybe more natural definition of hyperstructure. Let $\BBB=(X,P,B)$ be a ball structure. Then define its {\em hyper-ball structure} as $H(\BBB)=(\mathcal P(X),P,H(B))$, where
%$$
%H(B)(Y,\alpha)=\{Z\subseteq X\mid Z\subseteq B(Y,\alpha)\}=\widehat{B(Y,\alpha)},
%$$
%for every $Y\subseteq X$ and every $\alpha\in P$. Note that $H(\BBB)$ is never upper symmetric, except for $\BBB$ being the empty ball structure, since $\emptyset\in H(B)(X,\alpha)$, while $H(B)(\emptyset,\alpha)=\{\emptyset\}$, for every $\alpha\in P$. Moreover, if $\BBB$ is a quasi-ballean, then $H(\BBB)$ is a quasi-ballean and, in this case we call it {\em hyper-quasi-ballean}. 

%The relationship between those two hyperstructures is the that $\exp\BBB=\Sym(H(\BBB))$, for every ball structure $\BBB$.
We conclude this discussion by very briefly relating to the classical theory of quasi-uniform spaces. The definition of a quasi-uniform structure on the power set of a quasi-uniform space is a classical construction, and it is very similar to the entourage hyperstructure we have just introduced. A wide introduction and a broad bibliography can be found in \cite{Kun}. Moreover, we refer to \cite{Isb} for an analogue of the semi-coarse hyperstructure in the framework of uniform spaces.
%\end{esempio}

%We anticipate here a very important example of a coarse structure, that will be useful later in this paper (see \S\ref{sub:coarse_lst}).

\subsection{Finitely generated monoids}\label{sub:fin_gen_mon}

A {\em magma} is a pair $(M,\cdot)$, where $M$ is a set and $\cdot\colon M\times M\to M$ is a map. A magma $(M,\cdot)$ is called {\em unitary} if there exists a {\em neutral  element} $e\in M$ such that $g\cdot e=e\cdot g=g$, for every $g\in M$. A unitary magma is a {\em monoid}, if $\cdot$ is associative. 

Let $M$ be a monoid. We say that $M$ is {\em finitely generated} if there exists a finite subset $\Sigma$ of $M$ such that, for every $g\in M$ there exist $n\in\N$ and $\sigma_1,\dots,\sigma_n\in\Sigma$ which satisfy $g=\sigma_1\cdots\sigma_n$.

In this subsection we want to briefly discuss the existence of precisely two inner quasi-coarse structures on a finitely generated monoid (see Proposition \ref{prop:fin_gen_mon}). The proof we give is similar to the case of finitely generated groups (see, for example, \cite{harpe}).

Let $M$ be a monoid which is finitely generated by $\Sigma$. Let us define the {\em (left) word extended quasi-metric $d_\Sigma^\lambda$} as follows: for every pair of elements $x,y\in M$,
$$
d_\Sigma^\lambda(x,y)=%\begin{cases}\begin{aligned}&
\min\{n\mid\exists\sigma_1,\dots,\sigma_n\in\Sigma:y=x\sigma_1\cdots\sigma_n\}%\quad&\text{if the minimum exists,}\\
%&\,\infty&\text{otherwise.}\end{aligned}\end{cases}
$$
(we denote $\min\emptyset=\infty$). The map $d_\Sigma^\lambda\colon M\times M\to\N\cup\{\infty\}$ is actually an extended quasi-metric and thus $M$ can be endowed with the metric entourage structure $\E_{d_\Sigma^\lambda}$, which is a quasi-coarse structure. Moreover, note that $d_\Sigma^\lambda$ is {\em left-non-expanding}, i.e., for every $x,y,z\in M$, $d_\Sigma^\lambda(zx,zy)\leq d_\Sigma^\lambda(x,y)$. If $M$ is a finitely generated group, then $d_\Sigma^\lambda$ is {\em left-invariant} (i.e., for every $x,y,z\in M$, $d_\Sigma^\lambda(zx,zy)= d_\Sigma^\lambda(x,y)$). Similarly, one can define a right-non-expanding extended quasi-metric $d_\Sigma^\rho$ on $M$, called {\em(right) word quasi-metric}: for every $x,y\in M$,%, defined by the following law: for every $x,y\in M$,
$$
d_\Sigma^\rho(x,y)=%\begin{cases}\begin{aligned}&
\min\{n\mid\exists\sigma_1,\dots,\sigma_n\in\Sigma:y=\sigma_1\cdots\sigma_nx\}.%\quad&\text{if the minimum exists,}\\
%&\,\infty&\text{otherwise.}\end{aligned}\end{cases}
$$

\begin{proposizione}\label{prop:fin_gen_mon}
	Let $M$ be a monoid and $\Sigma$ and $\Delta$ be two finite subsets of $M$ which generate the whole monoid. Then $\E_{d^\lambda_\Sigma}=\E_{d^\lambda_\Delta}$ and $\E_{d^\rho_\Sigma}=\E_{d^\rho_\Delta}$.
\end{proposizione}
\begin{proof}
	Define $k=\max\{d^\lambda_\Delta(e,\sigma)\mid\sigma\in\Sigma\}$ and $l=\max\{d^\lambda_\Sigma(e,\delta)\mid\delta\in\Delta\}$. Let $x,y\in M$, suppose that $d^\lambda_\Sigma(x,y)=n$ and let $\sigma_1,\dots,\sigma_n\in\Sigma$ such that $y=x\sigma_1\cdots\sigma_n$. Suppose that $\sigma_i=\delta_{i,1}\cdots\delta_{i,k_i}$, for every $i=1,\dots,n$, where $k_i=d_\Delta^\lambda(e,\sigma_i)$ and $\delta_{i,j}\in\Delta$, for every $i=1,\dots,n$ and $j=1,\dots,k_i$. Then
	$$
	y=x\sigma_1\cdots\sigma_n=x\delta_{1,1}\cdots\delta_{1,k_1}\delta_{2,1}\cdots\delta_{n,k_n}
	$$
	and so $d^\lambda_\Delta(x,y)\le\sum_{i=1}^nk_i\le nk=kd^\lambda_\Sigma(x,y)$. Hence, $\E_{d_\Sigma^\lambda}\subseteq\E_{d_{\Delta}^\lambda}$. Similarly, $d^\lambda_\Sigma(x,y)\le ld^\lambda_\Delta(x,y)$ and then $\E_{d_\Delta^\lambda}\subseteq\E_{d_{\Sigma}^\lambda}$. A similar proof shows that $\E_{d_\Sigma^\rho}=\E_{d_{\Delta}^\rho}$.
\end{proof}
%Hence every finitely generated monoid has two inner quasi-coarse structures, similarly to finitely generated groups.

It is possible to extend the notion of Cayley graph, which is a useful tool to represent a finitely generated group, in the framework of finitely generated monoids. Let $M$ be a monoid and $\Sigma\subseteq M$ a finite subset which generates $M$. Then the ({\em left}) {\em Cayley graph of $M$ associated to $\Sigma$} is the directed graph $\Cay^\lambda(M,\Sigma)=(M,E)$, where $(x,y)\in E$ if and only if there exists $\sigma\in\Sigma$ such that $y=x\sigma$ or, equivalently, $d^\lambda_\Sigma(x,y)=1$. Similarly $\Cay^\rho(M,\Sigma)$, the ({\em right}) {\em Cayley graph}, can be constructed. The quasi-coarse space $(M,\E_{d^\lambda_\Sigma})$ and the graphic quasi-coarse structure on $\Cay^\lambda(M,\Sigma)$ (see \S\ref{sub:oriented_graph}) are asymorphic. Similarly, $(M,\E_{d^\rho_\Sigma})$ and the graphic quasi-coarse structure on $\Cay^\rho(M,\Sigma)$ are asymorphic.

%\begin{remark}[FINO A QUI, DA SPOSTARE A DOPO]
	 %Similarly, $\mathcal S_M^\rho$ is uniformly equi-bounded preserving, whenever $M$ is endowed with the quasi-coarse structure $\E_{d_\Sigma^\rho}$. Hence, Proposition \ref{prop:ideals_vs_entou_equibound} implies that $\II=\beta_{\E_{d_\Sigma^\lambda}}(e)$ is a monoid ideal and, moreover, $\E_\II^\lambda\subseteq\E_{d^\lambda_\Delta}$.
%	Hence, by applying Proposition \ref{prop:from_quasi-ballean_to_monoid_ideal}, $\II_{\BBB_{d^l_\Sigma}}=\{B_{d_\Sigma^l}(e,n)\mid n\in\N\}$ is a monoid ideal and $\BBB_{\II_{\BBB^l_\Sigma}}\prec\BBB_{d^l_\Sigma}$. Note that $\II_{\BBB_{d^l_\Sigma}}\subseteq[M]^{<\omega}$, which is the group ideal which defines the usual geometry on a finite generated group.
%\end{remark}

\subsection{Entourage structures on certain algebraic structures}\label{sub:algebraic}

In the classical coarse geometry, the geometry on finitely generated groups introduced via word metrics can be generalised to groups which are not finitely generated (see, for example, \cite{Sketchgroup,NicRos}). In this subsection we want to give a quick presentation of ways to introduce entourage structures on particular, more general algebraic structures.

A unitary magma $(M,\cdot)$ is a {\em loop} if, for every $a,b\in M$ there exist a unique $x\in M$ and a unique $y\in M$ such that
\begin{equation}\label{eq:quasigroup}
a\cdot x=b\quad\mbox{and}\quad y\cdot a=b.
\end{equation}
Since $e\cdot e=e$, \eqref{eq:quasigroup} implies that $e$ is the only neutral element. By \eqref{eq:quasigroup}, for every $g\in M$, there exist two elements $g^\rho, g^\lambda\in M$ such that $g\cdot g^\rho=e$ and $g^\lambda\cdot g=e$. Note that $(g^\rho)^\lambda=(g^\lambda)^\rho=g$, for every $g\in M$ (in fact, $(g^\rho)^\lambda\cdot g^\rho=e$, $g^\lambda(g^\lambda)^\rho=e$, $g\cdot g^\rho=e$, and $g^\lambda\cdot g=e$ and the conclusions follow by uniqueness of the solution of \eqref{eq:quasigroup}). A loop $(M,\cdot)$ has {\em right inverse property} if, for every $g,h\in M$, $(g\cdot h)\cdot h^\rho=g$. Similarly, a loop $(M,\cdot)$ has {\em left inverse property} if, for every $g,h\in M$, $g^\lambda\cdot(g\cdot h)=h$. A loop has {\em inverse property} if it has both right and left inverse property. A loop $M$ is said to have {\em two-side inverses} if $g^\lambda=g^\rho$, for every $g\in M$, and, in this case, we denote the inverse of $g$ by $g^{-1}$.

A \emph{unitary submagma} $N$ of a unitary magma $M$ is a subset $N\subseteq M$ that contains the identity of $M$ and it is closed under the operation. A unitary submagma $N$ is a \textit{subloop} of a loop $M$ if, for every parameters in $N$, the solutions of \eqref{eq:quasigroup} belongs to $N$. A unitary submagma $N$ of a monoid $M$ is called a \textit{submonoid}.

%Note that, for every loop $M$ and every element $g\in M$, $(g^\lambda)^\rho=g=(g^\rho)^\lambda$. In fact, $g^\lambda(g^\lambda)^\rho=e$ and $(g^\rho)^\lambda g^\rho=e$ and those equalities imply the claim because of the uniqueness of the solutions of the equations \eqref{eq:quasigroup}.

\begin{definizione}
	Let $M$ be an unitary magma and $\II$ be a family of subsets of $M$.
	\begin{compactenum}[$\bullet$]
		\item $\II$ is a {\em magmatic ideal} if it is an ideal on $M$ such that $\{e\}\in\II$ and, for every $\{x\},\{y\}\in\II$, $\{xy\}\in\II$.
		\item If $M$ is a monoid, $\II$ is a {\em monoid ideal} if it is a magmatic ideal and, for every $H,K\in\II$, $H\cdot K=\{h\cdot k\mid h\in H,\,k\in K\}\in\II$.
		\item If $M$ is a loop, a magmatic ideal $\II$ is a {\em right loop ideal} if, for every $F\in\II$, $F^\rho=\{g^\rho\mid g\in F\}\in\II$, $\II$ is a {\em left loop ideal} if, for every $F\in\II$, $F^\lambda=\{g^\lambda\mid g\in F\}\in\II$, and $\II$ is a {\em loop ideal} if it is both a left loop ideal and a right loop ideal. 
		\item If $M$ is a group, $\II$ is a {\em group ideal} (\cite{ProZar}) if it is both a monoid ideal and a loop ideal.%such that, for every $F\in\II$, $F^{-1}=\{g^{-1}\mid g\in F\}\in\II$.
	\end{compactenum}
\end{definizione}
%A {\em base for a magmatic ideal} on a unitary magma $M$ is a family $\BB\subseteq\mathcal P(M)$ such that $\widehat\BB$ is a magmatic ideal. Similarly, the concepts of bases for a left loop ideal, for a right loop ideal, for a monoid ideal, and for a group ideal can be defined.

%\begin{remark}\label{rem:magmatic_ideals}
%	(a) Directed unions of magmatic ideals (left or right loop ideals, monoid ideals, or group ideals) are magmatic ideals (left or right loop ideals, monoid ideals, or group ideals, respectively). Similarly, arbitrary intersections or magmatic ideals (left or right loop ideals, monoid ideals, or group ideals) are magmatic ideals (left or right loop ideals, monoid ideals, or group ideals, respectively).
	
%	(b)
If $\II$ is a magmatic ideal on a unitary magma $M$, then $\bigcup\II$ is a unitary submagma. Similarly, $\bigcup\II$ is a subloop (a submonoid, or a subgroup) if $\II$ is a loop ideal (a monoid ideal, or a group ideal, respectively). Moreover, if $\II$ is a right loop ideal on a loop with right inverse property (a left loop ideal on a loop with left inverse property), then $\bigcup\II$ has right inverse property (left inverse property, respectively).
	
%	(c) Let $M$ be a unitary magma and $\II$ be a magmatic ideal. Then $\II^2=\II\cdot\II=\{F\cdot K\mid F,K\in\II\}$ is a base for a magmatic idea. In fact, for every $F,F^\prime,K,K^\prime\in\II$, $(F\cdot K)\cup(F^\prime\cdot K^\prime)\subseteq (F\cup F^\prime)\cdot(K\cup K^\prime)\in\II^2$. Moreover, since $\{e\}\in\II$, $\II\subseteq\II^2$. By induction, we can show that, for every $n\in\N$, $\II^n=\II^{n-1}\II$ is a base for a magmatic ideal and $\II\subseteq\II^2\subseteq\cdots\subseteq\II^n\subseteq\cdots$. Define $\II^\infty=\bigcup_n\II^n$, which is a base for a magmatic ideal.
	
%	(d) If $M$ is a loop and $\II$ is a family of subsets of $M$, denote $\II^\rho=\{K^\rho\mid K\in\II\}$ and $\II^\lambda=\{K^\lambda\mid K\in\II\}$. If $\II$ is a left loop ideal, then both $\II^\rho$ and $\II^\lambda$ are left loop ideals. In fact, for every $K\in\II$, $(K^\lambda)^\lambda\in\II^\lambda$ and $(K^\rho)^\lambda=K=(K^\lambda)^\rho\in\II^\rho$, since $K^\lambda\in\II^\lambda$. Similarly, if $\II$ is a right loop ideal, then both $\II^\lambda$ and $\II^\rho$ are right loop ideals.
	
%	(e) If $M$ is a group and $\II$ is a monoid ideal, then $\II^{-1}=\{K^{-1}\mid K\in\II\}$ is a monoid ideal itself.
%\end{remark}

The leading examples of magmatic ideals are the finitary magmatic ideal, the finitary monoid ideal, the finitary (left, right) loop ideal, and the finitary group ideal.
\begin{example}
	Let $M$ be an unitary magma, then the family $\II=[M]^{<\infty}=\{F\subseteq M\mid \text{$M$ is finite}\}$ is a magmatic ideal, called {\em finitary magmatic ideal}. If $M$ is a monoid, then $\II$ is a monoid ideal, called {\em finitary monoid ideal}. If $M$ is a loop with (right) inverse property, then $\II$ is a (right) loop ideal, called {\em finitary (right) loop ideal}. Finally, if $M$ is a group, then $\II$ is a group ideal, called {\em finitary group ideal}.% and $\BBB_\II$ is called {\em finitary group ballean}.
\end{example}

Let $M$ be a unitary magma and $\II$ a magmatic ideal on $M$. If $A$ is a subset of $M\times M$, define $M\cdot A=\{(mx,my)\mid m\in M,\,(x,y)\in A\}$. Then we can define the {\em left magmatic entourage structure} $\E_{\II}$ on $M$ as follows:
$$
\E^\lambda_{\II}=\widehat{\{E_I^\lambda\mid I\in\II\}},\quad\mbox{where, for every $I\in\II$,}\quad E_I^\lambda=M(\{e\}\times I)=\{(x,xk)\mid x\in M,k\in I\}.
$$
Similarly, we can define the {\em right magmatic entourage structure} $\E_{\II}^\rho$, where the action of $M$ is on the right:
$$
\E^\rho_{\II}=\widehat{\{E_I^\rho\mid I\in\II\}},\quad\mbox{where, for every $I\in\II$,}\quad E_I^\rho=(\{e\}\times I)M.
$$

%\begin{proposition}\label{prop:magma_inverse_transit}
%	Let $M$ be a unitary magma and $\II$ be a magmatic ideal.
%	\begin{compactenum}[(a)]
%		\item If $M$ is a loop with the right inverse property, then $(\E_\II^\lambda)^{-1}\subseteq\E_{\II^\rho}^\lambda$. Similarly, if $M$ is a loop with the left inverse property, then $(\E_{\II}^\rho)^{-1}\subseteq\E_{\II^\lambda}^\rho$.
%		\item If $M$ is a monoid, then $\Wfun(\E_{\II}^\lambda)=\E_{\mathcal J}^\lambda$, where $\mathcal J=\widehat{\II^\infty}$ (see Remark \ref{rem:magmatic_ideals}(c)). Similarly, $\Wfun(\E_{\II}^\rho)=\E_{\mathcal J}^\rho$.
%	\end{compactenum}
%\end{proposition}
%\begin{proof}
%	(a) Let $K\in\II$. Then, for every $(x,xk)\in E_K^\lambda$, where $x\in M$ and $k\in K$, $(xk,x)=(xk,(xk)k^\rho)\in E_{K^\rho}^\lambda$, provided that $M$ has the right inverse property. Hence $(E_K^\lambda)^{-1}\subseteq E_{K^\rho}^\lambda\in\E_{\II^\rho}^\lambda$. Similarly, $(\E_{\II}^\rho)^{-1}\subseteq\E_{\II^\lambda}^\rho$ if $M$ has the left inverse property.
	
%	(b) First of all, we claim that $\E_{\mathcal J}^\lambda$ is a quasi-coarse structure. Let $F=F_1\cdots F_m,K=K_1\cdots K_n\in\mathcal J$. Then $E_F^\lambda\circ E_K^\lambda\subseteq E_{FK}^\lambda$ and $FK\in\mathcal J$. In fact, for every $x\in M$, $y\in F$, $z\in K$, then $(x,xy)\circ(xy,xyz)=(x,x(yz))\in E_{FK}$. The same computation shows that $\E_{\mathcal J}^\lambda\subseteq \Wfun(\E_\II^\lambda)$ and thus the claim follows from Theorem \ref{theo:refl_corefl}. The last claim can be similarly deduced.
%\end{proof}

\begin{proposition}Let $M$ be a unitary magma and $\II$ a magmatic ideal on $M$.
	\begin{compactenum}[(a)]
		\item $\E_{\II}^\lambda$ and $\E_{\II}^\rho$ are entourage structures.
		\item If $M$ is a loop with right inverse property and $\II$ is a right loop ideal (with the left inverse property and $\II$ is a left loop ideal), then $\E_{\II}^\lambda$ is a semi-coarse structure, called {\em left loop semi-coarse structure} ($\E_{\II}^\rho$ is a semi-coarse structure, called {\em right loop semi-coarse structure}, respectively). 
		\item If $M$ is a monoid and $\II$ is a monoid ideal, then $\E_{\II}^\lambda$ and $\E_{\II}^\rho$ are quasi-coarse structures, called {\em left monoid quasi-coarse structure} and {\em right monoid quasi-coarse structure}, respectively.
		\item If $M$ is a group, then $\E_{\II}^\lambda$ and $\E_{\II}^\rho$ are coarse structures, called {\em left group coarse structure} and {\em right group coarse structure}, respectively.
	\end{compactenum}
\end{proposition}
\begin{proof}
	Item (a) is trivial and item (d) has already been proved (\cite{NicRos}). 
	
	(b) Let $M$ be a loop with right inverse property and $\II$ be a right loop. Let $K\in\II$. Then, for every $(x,xk)\in E_K^\lambda$, where $x\in M$ and $k\in K$, $(xk,x)=(xk,(xk)k^\rho)\in E_{K^\rho}^\lambda$. Hence $(E_K^\lambda)^{-1}\subseteq E_{K^\rho}^\lambda\in\E_{\II}^\lambda$. We can prove similarly the other assertion.
	
	Finally, item (c) follows from the observation that, for every $F,K\in\II$, $E_F^\lambda\circ E_K^\lambda\subseteq E_{FK}^\lambda$.
\end{proof}

Of course, if $M$ is an abelian unitary magma and $\II$ is a magmatic ideal, then $\E_{\II}^\lambda=\E_{\II}^\rho$. Note that, if $M$ is an abelian loop with the right inverse property, then it has the inverse property. In the next remark we discuss a situation in which the left and the right magmatic entourage structure are asymorphic, even though the may not be equal.

\begin{remark}
	Let $G$ be a group and $\II$ be a magmatic ideal. Note that $\II^{-1}=\{K^{-1}\mid K\in\II\}$ is still a magmatic ideal and, more precisely, if $\II$ is a loop ideal or a monoid ideal, then so it is $\II^{-1}$. %Then, because of Proposition \ref{prop:magma_inverse_transit}(a),
	%$$
	%\E_\II^\lambda=((\E_{\II}^\lambda)^{-1})^{-1}\subseteq(\E_{\II^{-1}}^\lambda)^{-1}\subseteq\E_{(\II^{-1})^{-1}}^\lambda=\E_\II^\lambda
	%$$
	%and thus the previous inclusions are equalities. Similarly, $(\E_\II^\rho)^{-1}=\E_{\II^{-1}}^\rho$.
	Consider the map $i\colon G\to G$ such that $i(g)=g^{-1}$, for every $g\in G$. Then $i\colon(G,\E_\II^\lambda)\to(G,\E_{\II^{-1}}^\rho)$ is an asymorphism. In fact, for every $K\in\II$ and every $(x,xk)\in E_K^\lambda$, $(i\times i)(x,xk)=(x^{-1},k^{-1}x^{-1})\in E_{K^{-1}}^\rho$. The same conclusion holds if $\II$ is a loop ideal, a monoid ideal, or a group ideal. In particular, if $\II$ is a loop ideal or a group ideal, then $\II=\II^{-1}$ and thus $(G,\E_\II^\lambda)$ and $(G,\E_{\II}^\rho)$ are asymorphic. Hence, on a group $G$, if $\II$ is a loop ideal, we simply write $\E_\II$ instead of both $\E_\II^\lambda$ and $\E_{\II}^\rho$ when there is no risk of ambiguity.
\end{remark}

%\begin{remark}\label{rem:left_right_lst_coincide}
%	If $M$ is a unitary magma and $\II$ is a magmatic ideal on $M$, then the large-scale topologies $\beta_{\E_\II^\lambda}$ and $\beta_{\E_\II^\rho}$ coincide. In particular, 
%	$$
%	\beta_{\E_{\II}^\lambda}(e)=\{K\in\II\mid e\in K\}=\beta_{\E_{\II_M}^\rho}(e).
%	$$
%	Thus, we just write $\beta_{\II}^{unif}$ instead of $\beta_{\E_\II^\lambda}$ or $\beta_{\E_\II^\rho}$. Note that this notation is coherent with the one used in Example \ref{ex:monoid_ideals}.
%\end{remark}

A family of maps $\mathcal F$ between two entourage spaces $(X,\E_X)$ and $(Y,\E_Y)$ is {\em equi-bornologous} if, for every $E\in\E_X$, there exists $F\in\E_Y$ such that $(f\times f)(E)\subseteq F$, for every $f\in\mathcal F$. 

If $M$ is an unitary magma and $\E$ is an entourage structure on it, we define the following families of maps, which are the \textit{left} and the \textit{right shifts} in $M$:
\begin{equation}\label{eq:shifts}
\mathcal S_M^\lambda=\{s_x^\lambda\mid x\in M\}\quad\mbox{and}\quad\mathcal S_M^\rho=\{s_x^\rho\mid x\in M\},\text{ where, for every $x,y\in M$, $s_x^\lambda(y)=xy$ and $s_x^\rho(y)=yx$}.
\end{equation} 

%Those families of maps are sometimes equi-bornologous.
\begin{remark}\label{rem:equi-borno}
\begin{compactenum}[(a)]%\begin{fact}
\item For every finitely generated monoid $M$, if $\Sigma$ is a finite generating set, then the quasi-coarse space $(M,\E_{d^\lambda_\Sigma})$ makes the family $\mathcal S_M^\lambda=\{s_x^\lambda\mid x\in M\}$ equi-bornologous, since $d_\Sigma^\lambda$ is left-non-expanding.
\item Let $M$ be a monoid and $\II$ be a monoid ideal on $M$. Then $\mathcal S_M^\lambda$ and $\mathcal S_M^\rho$ are equi-bornologous if $M$ is endowed with the left monoid quasi-coarse structure $\E_\II^\lambda$ and the right monoid quasi-coarse structure $\E_\II^\rho$, respectively. In fact, let $e\in K\in\II$. Then, for every $x\in M$ and every $(y,yk)\in E_K^\lambda$, 
$$
(s_{x}^\lambda\times s_{x}^\lambda)(y,yk)=(xy,xyk)=xy(e,k)\in E_K^\lambda.
$$
%\end{proof}
\item The shifts of the unitary magmas $(\Z\cup\{e\},-,e)$ and $(\Q\cup\{e\},/,e)$ are equi-bornologous, once those magmas are equipped with the finitary magmatic entourage structure. Those result follow from a more general statement.
%\begin{proposizione}\label{prop:uniform_weak_associativity}
\item Let $M$ be a unitary magma such that there exists a map $r\colon M\to M$ with the property that $a(bc)=(ab)r(c)$, for every $a,b,c\in M$. Hence, we claim that, for every magmatic ideal $\II$ on $M$ such that $r(\II)=\{r(K)\mid K\in\II\}\subseteq\II$, the family of left shifts is equi-bornologous.
%\end{proposizione}
%\begin{proof}
	Let $e\in K\in\II$ be a generic element of the group ideal. Then, for every $a,b\in M$,
	$$
	s_a^\lambda(B_\II(b,K))=s_a^\lambda(bK)=a(bK)=(ab)r(K)=s_a^\lambda(b)r(K)=B_\II(s_a^\lambda(b),r(K)),
	$$
	which concludes the proof, since $r(K)\in\II$ and $e=r(e)\in r(K)$.
%\end{proof}
%In particular, for every unitary magma with the uniform weak associative property, the family of all left shifts are uniformly bornologous, once the unitary magma is endowed with the finitary ball structure.
\end{compactenum}
\end{remark}
The next proposition shows the importance of having the families of lift (right) shifts, defined in \eqref{eq:shifts}, equi-bornologous. We state the result just for $\mathcal S_M^\lambda$, but similar conclusions hold also for $\mathcal S_M^\rho$.

\begin{proposition}\label{prop:ideals_vs_entou_equibound}
	Let $M$ be an unitary magma and $\E$ be an entourage structure over $M$ such that $\mathcal S_M^\lambda$ is equi-bornologous. Let $\II=\{E[e]\mid E\in\E\}$. Then:
	\begin{compactenum}[(a)]
		\item $\II$ is a magmatic ideal and $\E_\II^\lambda\subseteq\E$;
		\item if $M$ is a loop with the inverse property and two-side inverses and $\E$ is a semi-coarse structure, then $\II$ is a loop ideal and $\E_\II^\lambda=\E$;
		\item if $M$ is a monoid and $\E$ is a quasi-coarse structure, then $\II$ is a monoid ideal and $\E_\II^\lambda\subseteq\E$;
		\item if $M$ is a group and $\E$ is a coarse structure, then $\II$ is a group ideal and $\E_\II^\lambda=\E$.
	\end{compactenum}
\end{proposition}
\begin{proof}
	The first assertion of item (a) is trivial since $E[e]\cup F[e]=(E\cup F)[e]$, for every $E,F\in\E$. Let now $F\in\E$ be an arbitrary entourage, and so $F[e]$ be an arbitrary element of $\II$. %$K\in\II$ and, without loss of generality, we can assume that $K=F[e]$, for some $F\in\E$. 
	Since $\mathcal S_M^\lambda$ is equi-bornologous, there exists $F^\prime\in\E$ such that, for every $x\in M$, $(s_x^\lambda\times s_x^\lambda)(F)\subseteq E^\prime$. Then, for every $(x,y)\in E_{F[e]}^\lambda$, there exists $k\in F[e]$ such that $y=xk$. Then 
	$$
	(x,y)=(x,xk)=(s_x^\lambda\times s_x^\lambda)(e,k)\in (s_x^\lambda\times s_x^\lambda)(F)\subseteq F^\prime
	$$
	and so $E_{F[e]}^\lambda\subseteq F^\prime$. Hence, $\E^\lambda_\II\subseteq\E$.
	
	(b) Let $E\in\E$ and let $x$ be an arbitrary element of $E[e]$. Then,
	$$
	(e,x^{-1})=(s_{x^{-1}}^\lambda\times s_{x^{-1}}^\lambda)(x,e)\in(s_{x^{-1}}^\lambda\times s_{x^{-1}}^\lambda)(E^{-1})\subseteq F,
	$$
	where $F$ can be chosen independently from the choice of $x$ in $E[e]$, since $\mathcal S_M^\lambda$ is equi-bornologous. Thus $E[e]^{-1}\subseteq F[e]\in\II$. 
	
	Consider now an arbitrary entourage $E\in\E$. We want to show that there exists $F\in\E$, such that $E\subseteq E_{F[e]}^\lambda$. Let $(x,y)\in E$ and denote by $F\in\E$ an entourage such that $(s_{z}^\lambda\times s_{z}^\lambda)(E)\subseteq F$, for every $z\in M$. Then
	\begin{equation}\label{eq:*}
	(x,y)=(s_{x}^\lambda\times s_{x}^\lambda)(e,x^{-1}y),\quad\mbox{where}\quad (e,x^{-1}y)=(s_{x^{-1}}^\lambda\times s_{x^{-1}}^\lambda)(x,y)\subseteq F,
	\end{equation}
	and thus $(x,y)\in E_{F[e]}$. Note that in \eqref{eq:*} we used that $M$ has the inverse property and two-sided inverses.
	
	(c) Thanks to item (a), we only need to show that $\II$ is a monoid ideal. Take $E,F\in\E$ and consider $E[e]\cdot F[e]$. Let $x\in E[e]$ and $y\in F[e]$ be two arbitrary elements, which means that $(e,x)\in E$ and $(e,y)\in F$. Denote by $E^\prime\in\E$ an entourage such that $(s_{x}^\lambda\times s_{x}^\lambda)(F)\subseteq E^\prime$, for every $x\in M$. Then
	$$
	(e,xy)=(e,x)\circ(x,xy)\in E\circ (s_{x}^\lambda\times s_{x}^\lambda)(F)\subseteq E\circ E^\prime\in\E,
	$$
	which shows that $xy\in(E\circ E^\prime)[e]$, and thus $E[e]F[e]\subseteq(E\circ E^\prime)[e]\in\II$.
	
	Finally, item (d) descends from items (b) and (c).
\end{proof}
%A result similar to Proposition \ref{prop:ideals_vs_entou_equibound} can be stated also for the entourage structure $\E_{d_\Sigma^\rho}$.

\begin{remark}
Let $M$ be a monoid generated by a finite subset $\Sigma$. By Remark \ref{rem:equi-borno} and Proposition \ref{prop:ideals_vs_entou_equibound}, $\E_\II^\lambda\subseteq\E_{d_{\Sigma}^\lambda}$, where $\II$ is the family of all subsets of $(M,d_\Sigma^\lambda)$ \textit{bounded from $e$}, i.e., contained in some ball centred at $e$. More precisely, $\II=[M]^{<\infty}$. We claim that $\E_\II^\lambda=\E_{d_{\Sigma}^\lambda}$. Let $R\ge 0$ and define $F_R=\{\sigma_1\cdots\sigma_n\mid n\leq R,\,\sigma_i\in\Sigma,\forall i=1,\dots,m\}$. Then $F_R\in\II$. Moreover, if $d_\Sigma^\lambda(x,y)\leq R$, then $y\in xF_R$. Hence $E_R\subseteq E_{F_R}^\lambda$, which shows the desired equality. Similarly, $\E_\II^\rho=\E_{d_{\Sigma}^\rho}$.
\end{remark}

Let $f\colon M\to N$ be a map between two unitary magmas. Then $f$ is called a {\em homomorphism} if, for every $g,h\in M$, $f(gh)=f(g)f(h)$ and $f(e_M)=e_N$. Propositions \ref{prop:homo_borno} and \ref{prop:homo_eff_proper} are relaxed versions of classical results in the framework of coarse structures on groups (\cite{NicRos}).

\begin{proposition}\label{prop:homo_borno}
	Let $f\colon M\to N$ be a homomorphism between unitary magmas, and $\II_M$ and $\II_N$ be two magmatic ideals on $M$ and $N$, respectively. Then the following properties are equivalent:
	\begin{compactenum}[(a)]
		\item $f(\II_M)=\{f(K)\mid K\in\II\}\subseteq\II_N$;
		\item $f\colon(M,\E_{\II_M}^\lambda)\to(N,\E_{\II_N}^\lambda)$ is bornologous;
		\item $f\colon(M,\E_{\II_M}^\rho)\to(N,\E_{\II_N}^\rho)$ is bornologous.
	\end{compactenum}
\end{proposition}
\begin{proof}
The implications (b)$\to$(a) and (c)$\to$(a) are trivial. In fact, for every $K\in\II_M$, $f(K)=f(E_K^\lambda[e])\subseteq ((f\times f)(E_K^\lambda))[e]$. Let us now prove (a)$\to$(b), and (a)$\to$(c) can be similarly shown. Let $K\in\II$. Then, for every $(x,xk)\in E_K^\lambda$, $(f\times f)(x,xk)=(f(x),f(x)f(k))\in E_{f(K)}^\lambda$, and thus $(f\times f)(E_K^\lambda)\subseteq E_{f(K)}^\lambda\in\E_{\II_N}^\lambda$.
\end{proof}

\begin{fact}\label{fact:homo_loop}
	Let $M$ and $N$ be two loops and $f\colon M\to N$ be a homomorphism between them.
	\begin{compactenum}[(a)]
		\item $f(x)^\lambda=f(x^\lambda)$ and $f(x)^\rho=f(x^\rho)$, for every $x\in M$.
		\item $f(M)$ is a {\em subloop} of $N$.% (i.e., closed under products and the solutions of \eqref{eq:quasigroup} with parameters in $f(M)$ belong to $f(M)$).
		\item If $M$ has two-sided inverses, then $f(N)$ has also two-sided inverses.
	\end{compactenum}
\end{fact}
\begin{proof}
	(a) Let $x\in M$. Then
	$$
	f(x)^\lambda f(x)=e=f(e)=f(x^\lambda x)=f(x^\lambda)f(x)\quad\text{and}\quad f(x)f(x)^\rho=e=f(e)=f(xx^\rho)=f(x)f(x^\rho),
	$$
	and so the conclusion follows by uniqueness of the solutions of \eqref{eq:quasigroup}.
	
	(b) Let $f(a)$ and $f(b)$ be two elements in $f(M)$. Then there exists a unique $x\in M$ such that $ax=b$ and so $f(a)=f(x)f(b)$. Moreover, the solution $f(x)$ is unique since $N$ is a loop.
	
	(c) For every $f(x)\in N$, $f(x)^\lambda=f(x^\lambda)=f(x^\rho)=f(x)^\rho$.
\end{proof}

\begin{proposition}\label{prop:homo_eff_proper}
	Let $f\colon M\to N$ be a homomorphism between two loops with inverse properties, and $\II_M$ and $\II_N$ be two magmatic ideals on $M$ and $N$, respectively. Assume that $M$ has two-side inverses and $f$ is a homomorphism. Then the following properties are equivalent:
	\begin{compactenum}[(a)]
		\item $f^{-1}(\II_N)\subseteq\II_M$;
%		\item $f\colon(M,\E_{\II_M}^{\lambda})\to(N,\E_{\II_N}^{\lambda})$ is weakly uniformly bounded copreserving;
%		\item $f\colon(M,\E_{\II_M}^{\lambda})\to(N,\E_{\II_N}^{\lambda})$ is uniformly bounded copreserving;
		\item $f\colon(M,\E_{\II_M}^{\lambda})\to(N,\E_{\II_N}^{\lambda})$ is effectively proper.
		%\item $f\colon(M,\E_{\II_M}^{\rho})\to(N,\E_{\II_N}^{\rho})$ is effectively proper.
	\end{compactenum}
\end{proposition}
\begin{proof}
	Implication (b)$\to$(a) is trivial. In fact, for every $K\in\II_N$, $f^{-1}(K)=f^{-1}(E_K^\lambda[e])\subseteq((f\times f)^{-1}(E_K^\lambda))[e]$ and $(f\times f)^{-1}(E_K^\lambda)\in\II_M$. % (d)$\to$(b)$\to$(a) are trivial. Let us now focus on (a)$\to$(c), while the proof (a)$\to$(d) is similar. 
	Conversely, let $e\in K\in\II_N$ and $(x,y)\in(f\times f)^{-1}(E_K^\lambda)$. Then $f(y)\in f(x)K$, which implies that $f(x^{-1}y)\in K$. Hence, $x^{-1}y\in f^{-1}(K)$ and thus $(x,y)\in E_{f^{-1}(K)}^\lambda$.
\end{proof}

\section{Categories of entourage spaces}\label{sec:category}

A {\em concrete category} $(\XX,\Ufor)$ is a pair where $\XX$ is a category and $\Ufor\colon\XX\to\Set$ is a {\em faithful functor} (i.e., such that, for every $f,g\in\Mor_{\XX}(X,Y)$, $\Ufor f=\Ufor g$ if and only if $f=g$). In that situation, the {\em fiber} of a set $A$ is the family of all $X$ in $\XX$ such that $\Ufor X=A$. If $(\mathcal X,\Ufor)$ is a concrete category and $X$ and $Y$ are two objects of $\mathcal X$, a morphism $f\colon \Ufor X\to \Ufor Y$ is a {\em $\mathcal X$-morphism with respect to $X$ and $Y$} whenever there exists $\overline f\in\Mor_\XX(X,Y)$ such that $\Ufor\overline f=f$. 

A {\em source} in a category $\XX$ is a family (possibly a proper class) $\{f_i\colon X\to X_i\}_{i\in I}$ of morphisms of $\XX$. Its dual notion is the one of sink. A {\em sink} in a category $\XX$ is a family $\{f_i\colon X_i\to X\}_{i\in I}$ of morphisms of $\XX$.

Suppose now that $(\XX,\Ufor)$ is a concrete category. A source $\{f_i\colon X\to X_i\}_{i\in I}$ of $\XX$ is {\em initial} if, for every morphism $f\colon \Ufor A\to \Ufor X$ of $\Set$, such that $\Ufor f_i\circ \Ufor f\colon\Ufor A\to \Ufor X_i$ is an $\XX$-morphism, then $f$ is an $\XX$-morphism. An {\em initial lifting} of a source $\{f_i\colon A\to \Ufor X_i\}$ in $\Set$, where, for every $i\in I$, $X_i$ is an object of $\XX$, is an initial source $\{g_i\colon B\to X_i\}_{i\in I}$ of $\XX$ such that $\Ufor B=A$ and $\Ufor g_i=f_i$, for every $i\in I$.

\begin{definizione}[\cite{Bru}]\label{def:cat_top}
	A concrete category $(\XX,\Ufor)$ is {\em topological} if:\begin{compactenum}[(a)]
		\item $\Ufor$ is {\em amnestic} (i.e. $f=1_X$, whenever $f\colon X\to X$ is an isomorphism of $\XX$ such that $\Ufor f=1_{\Ufor X}$);
		\item $\Ufor$ is \emph{transportable} (i.e., for every object $A$ of $\XX$ and every isomorphism $h\colon \Ufor A\to X$ of $\Set$, there exists an object $B$ of $\XX$ and an isomorphism $f\colon A\to B$ of $\XX$ such that $\Ufor f=h$);
		\item constant maps are morphisms of $\XX$;
		\item $\Ufor$ has {\em small fibers} (i.e., the fibers are sets);
		\item every singleton of $\Set$ has a unique element in its fiber;
		\item every source $\{f_i\colon A\to UX_i\}_{i\in I}$ of $\Set$, has an initial lifting.
	\end{compactenum}
\end{definizione}

Let us define the concrete category $\EN$, whose objects are entourage spaces and whose morphisms are bornologous maps between them. Moreover, we consider three full subcategories of $\EN$, namely, $\SPCS$, whose objects are semi-coarse spaces, $\QPCS$, the subcategory of quasi-coarse spaces, and, finally, $\PCS$, the subcategory of coarse spaces. All of those categories $\XX$ are concrete and so they have {\em forgetful functors} (i.e., faithful functors) $\Ufor_\XX\colon\XX\to\Set$. Moreover, there are the following forgetful functors:

\begin{equation}\label{forgetful}
\xymatrix{
&\mbox{$\EN$}\\
\mbox{$\SPCS$}\ar[ur]^{\Ufor_{\SPCS,\EN}}& &\mbox{$\QPCS$}\ar[ul]_{\Ufor_{\QPCS,\EN}}\\
&\mbox{$\PCS$}\ar[ul]^{\Ufor_{\PCS,\SPCS}}\ar[ur]_{\Ufor_{\PCS,\QPCS}}\ar[uu]|-{\Ufor_{\PCS,\EN}}
}
\end{equation}
We write $\Ufor$ instead of both $\Ufor_{\XX,\mathcal Y}\colon \XX\to\mathcal Y$ and $\Ufor_{\XX}\colon\XX\to\Set$ if there is no risk of ambiguity.

\begin{teorema}\label{theo:top_cat}
The categories $\EN$, $\SPCS$, $\QPCS$ and $\PCS$ are topological.
\end{teorema}

The only thing we need to verify is that those categories allow for lifting of initial sources since the other requests can be easily checked. Let $f\colon X\to(Y,\E)$ be a map between a set and an entourage space. We define the {\em initial entourage structure} $f_\ast\E$ as the entourage structure over $X$ generated by the base $\{(f\times f)^{-1}(E)\mid E\in\E\}$. If $\E$ is a semi-coarse structure (quasi-coarse structure), then $f_\ast\E$ is a semi-coarse structure (quasi-coarse structure, respectively). Moreover, $f\colon(X,f_\ast\E)\to(Y,\E)$ is bornologous and effectively proper.

\begin{proof}[Proof of Theorem \ref{theo:top_cat}]
%Let $\mathcal X$ be one of the four categories and denote by $\mathfrak X(X)$ the lattice of all the objects of $\XX$ on a set $X$.
%The forgetful functor $\Ufor\colon\XX\to\mathbf{Set}$ is  {\em amnestic}, i.e., an isomorphism $f$ in $\XX$ is an identity whenever $\Ufor f$ is an identity. Moreover, $\Ufor$ is {\em transportable}, i.e.,  for any object $A\in\XX$ and any bijection (i.e., $\mathbf{Set}$-isomorphism) $h\colon \Ufor A \to X$ there exists another object $B$ of $\XX$ and an isomorphism $f\colon A \to B$ in $\XX$ with $\Ufor f=h$ (i.e., the structure of $A$ can be ``transported" via the bijection $h$). Obviously, a singleton admits a unique structure and the constant maps in $\XX$ are morphisms.  The {\em fibers of $\Ufor$ are small}, i.e., the collection $\mathfrak X(X)$ of structures making a given set $X$ a coarse space is a subset of $\mathcal P(\mathcal P(X\times X))$, so it is a set, not a proper class. 
%It remains to check that the functor $\Ufor \colon  \XX\to \mathbf{Set}$ allows for lifting initial sources \cite{brmmer}. Namely, if $X$ is a set, $\{(Y_i,\E_{i})\mid i\in I\}$ is a family of coarse spaces and $f_i\colon X \to Y_i$, $i\in I$, are maps, then the source $(f_i)$ has an initial lift along $\Ufor$. Namely, a structure $\E$ on $X$ such that $(X,\E)\in\XX$, all maps $f_i\colon (X,\E) \to (Y_i,\E_{i})$ are bounded preserving, and for every map $g\colon \Ufor Z \to X$, such that $f_i\circ g\colon Z \to (Y_i,\E_{i})$ is a bounded preserving map for every $i\in I$,  the map $g\colon Z \to  (X,\E)$ is bounded preserving. For every $i\in I$ let 
Let $\{f_i\colon X\to(Y_i,\E_i)\}_{i\in I}$ be a source of maps from a set to a family of entourage spaces. Define the entourage structure $\E$ over $X$ as $\E=\bigcap_{i\in I}(f_i)_\ast\E_i$. If $\E_i$ is a semi-coarse structure (a quasi-coarse structure), for every $i\in I$, then $\E$ is a semi-coarse structure (a quasi-coarse structure, respectively).
\end{proof}

A morphism $\alpha\colon X\to X^{\prime}$, in a category $\mathcal X$, is called:
\begin{compactenum}[$\bullet$]
	\item an {\em epimorphism} if every pair of morphisms $\beta,\gamma\colon X^{\prime}\to X^{\prime\prime}$ such that $\beta\circ\alpha=\gamma\circ\alpha$ satisfies $\beta=\gamma$;
	\item a {\em monomorphism} if every pair of morphisms $\beta,\gamma\colon X^{\prime\prime}\to X$ such that $\alpha\circ\beta=\alpha\circ\gamma$ satisfies $\beta=\gamma$;
	\item a {\em bimorphism} if it is both epimorphism and monomorphism.
\end{compactenum} 
%The following corollary is an immediate consequence of Theorem \ref{theo:top_cat} and the well known properties of topological categories
%\cite{DikTho}.
%\begin{corollario}\label{Coro:epi/mono}
%Let $\XX$ be $\EN$, $\SPCS$, $\QPCS$, or $\PCS$ and let $f$ be a morphism in $\XX$. 
%\begin{compactenum}[(i)]
%\item $f$ is an epimorphism in $\XX$ if and only if $f$ is surjective.
%\item $f$ is a monomorphism in $\XX$ if and only if $f$ is injective.
%\end{compactenum}
%\end{corollario}
%Let $\XX$ be $\EN$, $\SPCS$, $\QPCS$ or $\PCS$. The initial structure $f_\ast\E_Y$ of a single map $f\colon X \to (Y,\E_Y)$,  defined as in (\ref{lastEq}) with $f_i=f$, has an additional remarkable property. Namely, the map $f\colon (X,f_\ast\E_Y)\to(Y,\E_Y)$ is also effectively proper, as $R_f=\{(x,y)\in X\times X\mid f(x)=f(y)\}\inf_\ast\E_Y$. Therefore, if $f$ has large image in $Y$, then $f$ is a coarse equivalence. 
Thanks to Theorem \ref{theo:top_cat}, the fact that in those categories the epimorphisms are surjective morphisms and the monomorphisms are injective morphisms follows (\cite{DikTho}). In particular those four categories are not {\em balanced}. Recall that a category $\XX$ is {\em balanced} if every bimorphism is an isomorphism. The fact that $\PCS$ is topological and it is not balanced was already proved in \cite{DikZa}.

A category $\XX$ is {\em cowellpowered} if for every object $X$ and every source $\{e_i\colon X\to X_i\}_{i\in I}$ of epimorphisms (possibly a proper class), there exists a set $\{e_j\}_{j\in J}$, $J\subseteq I$, such that, for every $i\in I$, there exists $j\in J$ and an isomorphism $f$ of $\XX$ such that $e_i=f\circ e_j$. %A full subcategory $\XX$ of $\PCS$ is {\em epireflective} if it is closed under formation of subobjects and of products. This is not the usual definition of epireflective subcategory, but it is equivalent in this contest (see \cite{Cats}).

Since the epimorphisms of $\EN$, $\SPCS$, $\QPCS$, and $\PCS$ are surjective morphisms, those categories are cowellpowered. Moreover, in \cite{Zav} it was proved that every epireflective subcategory of $\PCS$ is cowellpowered. Hence the following question naturally arises.

\begin{question}
	Does there exist a subcategory of $\EN$ containing $\PCS$ which is not cowellpowered?
\end{question}

\subsection{Functors between the four categories $\EN$, $\SPCS$, $\QPCS$, and $\PCS$}\label{sub:fun_sym}

We want to study the relationships between $\EN$, $\SPCS$, $\QPCS$, and $\PCS$, and, in order to do that, we define some useful functors between the four categories, which will be summarised in the diagram \eqref{functors}. All these functors will only be defined on the objects, since the morphisms are ``fixed'' (i.e., if $\Ffun\colon\XX\to\mathcal Y$ is one of the functors that we are going to define and $f\colon X\to Y$ is a morphism of $\XX$, then $\Ufor(\Ffun f)=\Ufor f$).
\begin{compactenum}[$\bullet$]
\item $\Sym\colon\EN\to\SPCS$ is defined by the law $\Sym(X,\E)=(X,\Sym(\E))$, where $\Sym(\E)=\widehat{\{E\cap E^{-1}\mid E\in\E\}}=\E\cap\E^{-1}$, for every $(X,\E)\in\EN$. In a similar way, $\Sym\colon\QPCS\to\PCS$ is defined.% (see Lemma \ref{lemma:symmetrization} for its description in the framework of ball structures).
\item $\USym\colon\EN\to\SPCS$ is defined by the law $\USym(X,\E)=(X,\USym(\E))$, where $\USym(\E)=\widehat{\{E\cup E^{-1}\mid E\in\E\}}=\widehat{\E\cup\E^{-1}}$, for every $(X,\E)\in\EN$.
\item $\Wfun\colon\EN\to\QPCS$ is defined by the law $\Wfun(X,\E)=(X,\Wfun(\E))$, for every $(X,\E)\in\EN$, where $\Wfun(\E)=\widehat{\{E^n\mid n\in\N,\,E\in\E\}}$ and, for every $E\in\E$,
$$
E^n=\underbrace{E\circ\cdots\circ E}_{\text{$n$ times}}.
$$
Similarly, $\Wfun\colon\SPCS\to\PCS$ is defined.
\end{compactenum}
If we also consider the composite functor $\Wfun\circ \USym\circ \Ufor_{\QPCS,\EN}\colon\QPCS\to\PCS$, the situation can be represented by the following diagram:

\begin{equation}\label{functors}
\xymatrix{
&\mbox{$\EN$}\ar@/^/[dl]^{\Sym} \ar@/_/[dl]_{\USym} \ar[dr]^{\Wfun} \\
\mbox{$\SPCS$}\ar[dr]^{\Wfun} & & \mbox{$\QPCS$} \ar@/^/[dl]^{\quad\Wfun\circ \USym\circ\Ufor} \ar@/_/[dl]_{\Sym}\\
&\mbox{$\PCS$.} 
} 
\end{equation}
There is another endofunctor $\Jfun$ of $\EN$ that it is worth mentioning. Every entourage space $(X,\E)$ is associated to $\Jfun(X,\E)=(X,\E^{-1})$ and every morphism $f\in\Mor_{\EN}(X,Y)$ is fixed, i.e., $\Ufor f=\Ufor(\Jfun f)$. Since, for every entourage $E$ of $X$, $(f\times f)(E^{-1})=((f\times f)(E))^{-1}$, $\Jfun f$ is bornologous whenever $f$ is bornologous, and so $\Jfun$ is a functor. Note that $\Jfun|_{\SPCS}$ is the identity functor of $\SPCS$.
\begin{remark}
\begin{compactenum}[(a)]
	\item Note that the functor $\Sym$ generalises the definition of the semi-coarse hyperstructure from the entourage hyperstructure (see \S\ref{sub:hyper} for the definitions). More precisely, if $(X,\E)$ is an entourage space, then $\Sym(\mathcal P(X),\mathcal H(\E))=(\mathcal P(X),\exp\E)$. 
	\item It is not true in general that, if $X$ is a quasi-coarse space, $\exp X=\exp(\Sym X)$. In fact, let $X=\Z$ with the relation entourage structure induced by the usual order relation, which is a quasi-coarse structure. Then $\Sym X$ is $X$ endowed with the discrete coarse structure. However, $2\Z$ and $2\Z+1$ belong to the same connected component of $\exp X$.
\end{compactenum}
\end{remark}

In the following example we provide a semi-coarse structure that admits no maximal coarse structure that is finer than the original one.

\begin{esempio}
	Consider the following extended semi-metric $d$ on $\Z^2$: for every $(x,y),(z,w)\in\Z^2$,
	$$
	d((x,y),(z,w))=\begin{cases}
	\begin{aligned}&\lvert x-z\rvert&\text{if $y=w$,}\\
	&\lvert y-w\rvert&\text{if $x=z$,}\\
	&\infty&\text{otherwise.}\end{aligned}
	\end{cases}
	$$
	Then $(\Z^2,\E_d)$ is a semi-coarse space. We claim that there are two different maximal coarse-structures $\E_1$ and $\E_2$ on $\Z^2$ which are finer then $\E_d$. %This would imply that neither $\QPCS$ is co-reflective in $\EN$, nor $\PCS$ is co-reflective in $\SPCS$.
	
	Define the following extended metrics $d_1$ and $d_2$ as follows: for every $(x,y),(z,w)\in\Z^2$,
	$$
	d_1((x,y),(z,w))=\begin{cases}
	\begin{aligned}&\lvert x-z\rvert&\text{if $y=w$,}\\
	&\infty&\text{otherwise.}\end{aligned}
	\end{cases}\quad
	d_2((x,y),(z,w))=\begin{cases}
	\begin{aligned}
	&\lvert y-w\rvert&\text{if $x=z$,}\\
	&\infty&\text{otherwise.}\end{aligned}
	\end{cases}
	$$
	Then $\E_1=\E_{d_1}$ and $\E_2=\E_{d_2}$ satisfy the desired properties.
\end{esempio}

Let us now recall some other basic categorical definitions from \cite{Cats}. Let $\Ffun,\Gfun\colon\mathcal X\to\mathcal Y$ be two functors between two categories. A {\em natural transformation} $\eta$ from $F$ to $G$ (in symbols, $F\xrightarrow{\eta}G$) is a function that assign to each object $X$ of $\XX$, a morphism $\eta_X\colon\Ffun X\to\Gfun X$ of $\mathcal Y$ such that, for every other morphism $f\colon X\to X^\prime$ of $\XX$, the following diagram commutes:
$$
\xymatrix{
\Ffun X\ar^{\eta_X}[r]\ar_{\Ffun f}[d] &\Gfun X\ar^{\Gfun f}[d]\\
\Ffun X^\prime\ar^{\eta_{X^\prime}}[r] &\Gfun X^\prime.
}
$$
Let now $\Gfun\colon\XX\to\YY$ and $\Ffun\colon\YY\to\XX$ be two functors. Then $\Ffun$ is {\em co-adjoint for $\Gfun$} (or $\Ffun$ is {\em left adjoint of $\Gfun$}, or $\Ffun$ has a {\em right adjoint $\Gfun$}) and $\Gfun$ is {\em adjoint for $\Ffun$} (or $\Gfun$ is {\em right adjoint of $\Ffun$}, or $\Gfun$ has a {\em left adjoint $\Ffun$}) if there exist two natural transformations $\eta\colon id_\YY\to\Gfun\circ\Ffun$ (called {\em unit}) and $\varepsilon\colon\Ffun\circ\Gfun\to id_\XX$ (called {\em co-unit}) such that the following two {\em triangular identities} hold: for every object $X\in\XX$ and every $Y\in\YY$, the following triangles commutes:
$$
\xymatrix{
\Gfun X\ar^{\eta_{\Gfun X}}[r]\ar_{id_{\Gfun X}}[dr] & \Gfun\Ffun\Gfun X\ar^{\Gfun\varepsilon_X}[d]\\
& \Gfun X,
} 
\quad\quad\mbox{and}\quad\quad
\xymatrix{
\Ffun X\ar^{\Ffun\eta_Y}[r]\ar_{id_{\Ffun Y}}[dr] & \Ffun\Gfun\Ffun Y\ar^{\varepsilon_{\Ffun Y}}[d]  \\
&  \Ffun Y.
}
$$

Let $\mathcal Y$ be a full subcategory of a category $\XX$. Then $\YY$ is {\em reflective} in $\XX$ if there exists a functor $\Gfun\colon\XX\to\YY$, called {\em reflector}, which is a co-adjoint for the inclusion functor $\Ifun\colon\YY\to\XX$. Dually, $\YY$ is {\em coreflective} in $\XX$ if there exists a functor $\Gfun\colon\XX\to\YY$, called {\em co-reflector}, which is an adjoint for $\Ifun\colon\YY\to\XX$.

%Then $\mathcal Y$ is {\em reflective} in $\XX$ if the inclusion functor $\Ifun\colon\mathcal Y\to\XX$ has a left adjoint $\Gfun\colon\XX\to\mathcal Y$, called {\em reflector}. The dual concept is co-reflection: $\mathcal Y$ is {\em co-reflective} in $\XX$ if the inclusion functor $\Ifun\colon\mathcal Y\to\XX$ has a right adjoint $\Gfun\colon\XX\to\mathcal Y$, called {\em co-reflector}.

The functors previously defined can be used to prove the following theorem.

\begin{teorema}\label{theo:refl_corefl}
\begin{compactenum}[(a)]
\item $\QPCS$ is a reflective subcategory in $\EN$;
\item $\SPCS$ is a reflective and co-reflective subcategory in $\EN$;
\item $\PCS$ is a reflective subcategory in $\SPCS$;
\item $\PCS$ is a reflective and co-reflective subcategory in $\QPCS$.
\end{compactenum}
\end{teorema}
\begin{proof}
The thesis follows by proving the following assertions:
\begin{compactenum}[$\bullet$]
\item $\Wfun$ is a reflector of $\Ifun\colon\QPCS\to\EN$; 
\item $\USym$ is a reflector and $\Sym$ is a co-reflector of $\Ifun\colon\SPCS\to\EN$;
\item $\Wfun$ is a reflector of $\Ifun\colon\PCS\to\SPCS$;
\item $\Wfun\circ \USym\circ \Ufor_{\QPCS,\EN}$ is a reflector and $\Sym$ is a co-reflector of $\Ifun\colon\PCS\to\QPCS$;
\end{compactenum}
which are easy checks of the definitions.
\end{proof}
Embeddings of reflective subcategories preserve limits (e.g., products), while embeddings of co-reflective subcategories preserve colimits (e.g., coproducts and quotients). See \cite{Cats} for details. In \S\ref{sub:cat_constr}, we prove that neither $\QPCS$ is a co-reflective subcategory in $\EN$, nor $\PCS$ is a co-reflective subcategory in $\SPCS$ since they do not preserve some colimits.

\subsection{Product, coproducts and quotients}\label{sub:cat_constr}

Let $\XX$ be a category and $\{X_i\}_{i\in I}$ be a family of objects of $\XX$. A source $\{p_i\colon X\to X_i\}_{i\in I}$, where $X$ is an object of $\XX$, is the {\em product of $\{X_i\}_{i\in I}$ in $\XX$} if it satisfies the following universal property: for every other source $\{f_i\colon Y\to X_i\}_{i\in I}$, where $Y$ is another object of $\XX$, there exists a unique morphism $f\colon Y\to X$ such that $f_i=p_i\circ f$, for every $i\in I$.

Let $\{(X_i,\E_i)\}_{i\in I}$ be a family of entourage spaces. Let $X=\Pi_iX_i$ and $p_i\colon X\to X_i$, for every $i\in I$ be the projection maps. Then the {\em product entourage structure} $\E=\Pi_i\E_i$ is defined as
$$
\E=\widehat{\bigg\{\bigcap_{i\in I}(p_i\times p_i)^{-1}(E_i)\mid E_i\in\E_i,\,\forall i\in I\bigg\}}.
$$
We can check that $(X,\E)$ is the product in $\EN$. As we have already pointed out, since $\PCS$ is reflective in both $\SPCS$ and $\QPCS$, which are reflective in $\EN$, these categories are stable under taking products. Hence, the same construction leads to the product in $\SPCS$, $\QPCS$ and $\PCS$. The products in $\PCS$ are well-known objects (see, for instance, \cite{ProZar,DikZa}).

Let $\XX$ be a category and $\{X_k\}_{k\in I}$ be a family of objects of $\XX$. A sink $\{i_k\colon X_k\to X\}_{k\in I}$, where $X$ is an object of $\XX$, is the {\em coproduct of $\{X_k\}_{k\in I}$ in $\XX$} if it satisfies the following universal property: for every other sink $\{f_k\colon X_k\to Y\}_{k\in I}$, where $Y$ is another object of $\XX$, there exists a unique morphism $f\colon X\to Y$ such that $f_k=f\circ i_k$, for every $k\in I$.

Let $\{(X_k,\E_k)\}_{k\in I}$ be a family of entourage spaces. On the disjoint union $X=\bigsqcup_kX_k$ of the supports, we define the {\em coproduct entourage structure} $\E=\bigoplus_k\E_k$ as follows:
$$
\begin{gathered}
\E=\widehat{\{E_{J,\varphi}\mid J\in[I]^{<\infty},\,\varphi\colon J\to\bigcup_{k\in I}\E_k,\,\varphi(k)\in\E_k,\,\forall k\in I\}}\\
\mbox{and, for every such a $J$ and $\varphi$,}\quad E_{J,\varphi}=\Delta_X\cup\bigg(\bigcup_{j\in J}(i_j\times i_j)(\varphi(j))\bigg).
\end{gathered}
$$
It is not hard to check that $(X,\E)$ is actually the coproduct in $\EN$, $\SPCS$, $\QPCS$ and $\PCS$. However, note that we could not have concluded as in the product case that, once we proved that it is the coproduct of $\EN$, then it would automatically be the coproduct of the other categories. In fact, $\QPCS$ is not coreflective in $\EN$ and $\PCS$ is not coreflective in $\SPCS$. %Next, we give an example of colimit which changes if we move from one subcategory to another.

Let $(\XX,\Ufor)$ be a concrete category. Let $X$ be an object of $\XX$, $A$ be a set, and $f\colon\Ufor X\to A$ be an epimorphism in $\Set$ (i.e., a surjective map). Then the {\em quotient} of $f$ and $X$ is a morphism $\overline f\colon X\to Y$ of $\XX$ with $\Ufor \overline f=f$, and that satisfies the following universal property: for every other morphism $g\colon\Ufor Y\to\Ufor Z$ of $\Set$, $g$ is an $\XX$-morphism, provided that $g\circ f$ is an $\XX$-morphism. 

We want to construct quotients in the four categories we are considering. Let $q\colon(X,\E)\to Y$ be a surjective map from an entourage space to a set. Then the {\em quotient entourage structure} on $Y$ is $q(\E)=\{(q\times q)(E)\mid E\in\E\}$. Moreover, if $\E$ is a semi-coarse structure, then $(Y,q(\E))$ is a semi-coarse space and thus it is also the quotient structure in $\SPCS$. However, as proved in \cite{DikZa}, if $\E$ is a coarse structure, then $q(\E)$ is not a quasi-coarse structure in general. Then the quotient structure in $\QPCS$ (in $\PCS$) is $\overline\E^q$, where $\overline\E^q$ is the finest quasi-coarse structure (coarse structure, respectively) which contains $q(\E)$, namely $\Wfun((Y,q(\E)))$. Hence, in particular, $\QPCS$ is not co-reflective in $\EN$ and $\PCS$ is not co-reflective in $\SPCS$.

Moreover, \cite[Theorem 4.12]{DikZa} can be modified in this setting. Recall that a surjective map $q\colon(X,\E)\to Y$ from a quasi-coarse space to a set is {\em weakly soft} (\cite{DikZa}) if, for every $E\in\E$, there exists $F\in\E$ such that $E\circ R_q\circ E\subseteq R_q\circ F\circ R_q$. %It is worth mentioning also he following property: a surjective map $q\colon(X,\E)\to Y$ from a quasi-coarse space to a set is {\em uniformly soft} (in \cite{DikZa}, it is called {\em soft}) if, for every $E\in\E$, there exists $F\in\E$ such that $R_q\circ E\subseteq F\circ R_q$. A uniformly soft map is trivially weakly soft, since $R_q\circ R_q=R_q$, for every map $q\colon X\to Y$. Furthermore, if $q\colon(X,\E)\to Y$ is a uniformly soft map from an entourage space to a set, then $q\colon(X,\beta_{\E})\to Y$ is soft.
\begin{theorem}
Let $q\colon(X,\E)\to Y$ be a surjective map from a quasi-coarse space to a set. Then $q(\E)$ is a quasi-coarse structure, and thus $\overline{\E}^q=q(\E)$, if and only if $q$ is weakly soft. In particular, $q(\E)$ is the categorical quotient structure of $\QPCS$ (of $\PCS$, provided that $\E$ is a coarse structure) if and only if $q$ is weakly soft.
\end{theorem}
\begin{proof}
The proof that shows \cite[Theorem 4.12]{DikZa} can be easily adapted to obtain this result.
\end{proof}

%Suppose that $\{(X_i,\E_i)\}_{i\in I}$ be a family of entourage spaces. Then $\beta_{\Pi_i\E_i}=\Pi_i\beta_{\E_i}$ and $\beta_{\bigoplus_i\E_i}=\bigoplus_i\beta_{\E_i}$.

\section{Equivalence relations induced by functors}\label{sec:sym-c.e.}

A very important notion in coarse geometry is the one of {\em coarse equivalence} (\cite{Roe}). Let $f,g\colon X\to(Y,\E)$ be two maps from a set to a coarse space. Then $f$ and $g$ are {\em close}, and we denote this fact by $f\sim g$, if $\{(f(x),g(x))\mid x\in X\}\in\E$. Since $\E$ is a coarse space, then $\sim$ is an equivalence relation. A subset $Y$ of a coarse space $(X,\E)$ is {\em large} if there exists $E\in\E$ such that $E[Y]=\bigcup_{y\in Y}E[y]=X$. A map $f\colon(X,\E_X)\to(Y,\E_Y)$ between coarse spaces is a {\em coarse equivalence} if it is bornologous and there exists another bornologous map $g\colon Y\to X$ such that $g\circ f\sim id_X$ and $f\circ g\sim id_Y$.

Suppose that $\Ffun$ is a functor from a category $\mathcal X$ to $\PCS$. Then a notion of closeness can be inherited by $\mathcal X$ from $\PCS$. Let $f,g\colon X\to Y$ be two morphisms of $\mathcal X$. We say that $f$ is {\em $\Ffun$-close} to $g$ (and we write $f\sim_{\Ffun} g$) if $\Ffun f$ is close to $\Ffun g$ in $\PCS$. These new relations are equivalences. Moreover, a morphism $k\colon W\to Z$ of $\XX$ is a {\em $\Ffun$-coarse inverse} of a morphism $h\colon Z\to W$ if $k\circ h\sim_{\Ffun}id_Z$ and $h\circ k\sim_{\Ffun}id_W$. Thanks to this notion, we can define equivalences between the objects of $\EN$, $\SPCS$ and $\QPCS$. 
\begin{definizione}
Suppose that $\mathcal X$ is a subcategory of $\EN$ and $\Ffun$ is a functor from $\mathcal X$ to $\PCS$. A map $f\colon (X,\E_X)\to(Y,\E_Y)$ between two objects of $\mathcal X$ is a {\em $\Ffun$-coarse equivalence} if $f\colon (X,\E_X)\to(Y,\E_Y)$ is bornologous and it has a bornologous $\Ffun$-coarse inverse $g\colon (Y,\E_Y)\to(X,\E_X)$. In this case, $(X,\E_X)$ and $(Y,\E_Y)$ are called {\em $\Ffun$-coarsely equivalent}.
\end{definizione}
The concept just introduced induces an equivalence relation between objects of $\XX$.

In particular, in this section we focus our attention to the equivalence induced by the functor $\Sym\colon\QPCS\to\PCS$. In Theorem \ref{prop:sym-quasi-coarse-equivalence} we characterise $\Sym$-coarse equivalences. 

A map $f\colon(X,\E_X)\to(Y,\E_Y)$ between quasi-coarse spaces is {\em large-scale surjective} if $f(X)$ is large in $\Sym(Y,\E_Y)$. If $f$ is also large-scale injective, then it is {\em large-scale bijective}. The following proposition characterises large-scale bijective maps between quasi-coarse spaces.
\begin{proposition}
Let $f\colon(X,\E_X)\to(Y,\E_Y)$ be a map between quasi-coarse space. Then $f$ is large-scale bijective if and only if it has a $\Sym$-coarse inverse. In particular, every $\Sym$-coarse inverse is large-scale bijective.
\end{proposition}
\begin{proof}
($\to$) Let $M=M^{-1}\in\Sym(\E_Y)\subseteq\E_Y$ be an entourage such that $M[f(X)]=Y$. For every $y\in Y$, there exists $x_y\in X$ such that $(y,f(x_y))\in M$. If $y\in f(X)$, suppose that $x_y\in f^{-1}(y)$. Define $g\colon Y\to X$ with the following law: $g(y)=x_y$, for every $y\in Y$. Then $(f(g(y)),y)\in M$ for every $y\in Y$, which witnesses that $f\circ g\sim_{\Sym}id_Y$. The fact that $f$ is large-scale injective proves that $g\circ f\sim_{\Sym} id_X$. 

($\gets$) Let now $g\colon Y\to X$ be a $\Sym$-coarse inverse of $f$. Let $M=M^{-1}\in\E_X$ and $N=N^{-1}\in\E_Y$ be two entourages that demonstrate that $g\circ f\sim_{\Sym}id_X$ and $f\circ g\sim_{\Sym}id_Y$, respectively. Note that, for every $y\in Y$, $f(g(y))\in f(X)$ and $(y,f(g(y))),(f(g(y)),y)\in N$. Hence $f$ is large-scale surjective. Moreover, since $R_f\subseteq M\circ M$, $f$ is large-scale injective.

The last assertion is trivial since, if $g$ is a $\Sym$-coarse inverse of $f$, then $f$ is a $\Sym$-coarse inverse of $g$.
\end{proof}

\begin{proposition}\label{prop:ls_inj_sur_open}
Let $f\colon(X,\E_X)\to(Y,\E_Y)$ be a large-scale bijective map between quasi-coarse spaces and let $g$ be a $\Sym$-coarse inverse of $f$. %between quasi-coarse spaces is large-scale bijective, then it has a $\Sym$-coarse inverse. Moreover, 
Then, the following properties are equivalent:
\begin{compactenum}[(a)]
\item $f$ is bornologous;
\item $g$ is weakly uniformly bounded copreserving;
\item $g$ is uniformly bounded copreserving;
\item $g$ is effectively proper.
\end{compactenum}
Moreover, every other $\Sym$-coarse inverse $h$ of $g$ satisfies $h\sim_{\Sym}f$.
\end{proposition}
\begin{proof}
Since $g$ is large-scale injective, the equivalences (b)$\leftrightarrow$(c)$\leftrightarrow$(d) descend from Proposition \ref{prop:unif_bound_cop_eff_proper}. Suppose now that $f$ is bornologous. Let $E\in\E_X$ and consider $(g\times g)^{-1}(E)$. Denote by $M=M^{-1}$ the entourage of $\E_Y$ such that $(f(g(z)),z)\in M$, for every $z\in Y$. Then, for every $(x,y)\in(g\times g)^{-1}(E)$,
$$
(x,y)=(x,f(g(x)))\circ(f(g(x)),f(g(y)))\circ(f(g(y)),y)\in M\circ(f\times f)(E)\circ M\in\E_Y.
$$
Conversely, suppose that $g$ is effectively proper. Denote by $N=N^{-1}\in\E_X$ the entourage which witnesses that $g\circ f\sim_{\Sym}id_X$. Let $E\in\E_{X}$ and $(x,y)\in E$. Then
$$
(g(f(x)),g(f(y)))=(g(f(x)),x)\circ(x,y)\circ(y,g(f(y)))\in N\circ E\circ N\in\E_X
$$
and thus $(f(x),f(y))\in(g\times g)^{-1}(N\circ E\circ N)\in\E_Y$.

Finally, if $h$ is another $\Sym$-coarse inverse of $g$, then, for every $x\in X$, $(g(f(x)),g(h(x)))=(g(f(x)),x)\circ(x,g(h(x)))\in N\circ K$, where $K=K^{-1}\in\E_X$ is an entourage that shows that $g\circ h\sim_{\Sym}id_X$. Hence $(f(x),h(x))\in(g\times g)^{-1}(N\circ K)$ and so $f\sim_{\Sym}h$ since $(g\times g)^{-1}(N\circ K)=((g\times g)^{-1}(N\circ K))^{-1}\in\E_Y$.
\end{proof}
Note that, with an easy variation of the proof of Proposition \ref{prop:ls_inj_sur_open}, one can prove that every large-scale injective map $f\colon(X,\E_X)\to Y$ from a quasi-coarse space to a set has a {\em partial $\Sym$-coarse inverse}, i.e., a map $g\colon Y^\prime\to(X,\E_X)$, where $Y^\prime\subseteq Y$, such that $g\circ f\sim_{\Sym} id_Y$.

\begin{theorem}\label{prop:sym-quasi-coarse-equivalence}
Let $f\colon (X,\E_X)\to (Y,\E_Y)$ be a map between quasi-coarse spaces. Then the following are equivalent:
\begin{compactenum}[(a)]
\item $f$ is a $\Sym$-coarse equivalence;
\item $f\colon (X,\E)\to (Y,\E^\prime)$ is large-scale bijective, bornologous and weakly uniformly bounded copreserving;
\item $f\colon (X,\E)\to (Y,\E^\prime)$ is large-scale bijective, bornologous and uniformly bounded copreserving;
\item $f\colon (X,\E)\to (Y,\E^\prime)$ is large-scale surjective, bornologous and effectively proper;
\item there exist two subspaces $X^\prime\subseteq X$ and $Y^\prime\subseteq Y$, which are large in $(X,\Sym(\E_X))$ and in $(Y,\Sym(\E_Y))$, respectively, and an asymorphism $f^\prime\colon X^\prime\to Y^\prime$.
\end{compactenum}
\end{theorem}
\begin{proof}
The equivalences (b)$\leftrightarrow$(c)$\leftrightarrow$(d) follow from Proposition \ref{prop:eff_cop_wcop}.	
	
(a)$\to$(d) Fix and entourage $E\in\E_Y$ and let $(x,y)\in(f\times f)^{-1}(E)$ be an arbitrary element in its preimage. Let $M=\{(x,g(f(x))),(g(f(x)),x)\mid x\in X\}\in\E_X$. Then
$$
(x,y)=(x,g(f(x)))\circ(g(f(x)),g(f(y)))\circ(g(f(y)),y)\in M\circ(g\times g)(E)\circ M\in\E_X.
$$
Hence, $f$ is effectively proper. Moreover, $f$ is large-scale surjective, since $\Sym(\E_Y)\subseteq\E_Y$ and, if $N=\{(y,f(g(y)))\mid y\in Y\}$, then the entourage $N\cup N^\prime\in \Sym(\E_Y)\subseteq\E_Y$ witnesses the property.

(d)$\to$(e) Suppose that $f\colon X\to Y$ satisfies item (d). Let $X^\prime\subseteq X$ be a subset with the following property: for every $x\in X$, $\lvert X^{\prime}\cap f^{-1}(f(x))\rvert=1$. Then $f^\prime=f|_{X^\prime}\colon X^\prime\to Y^\prime$, where $Y^\prime=f(X)=f(X^\prime)$, is bijective. Moreover, $f^\prime\colon(X^\prime,\E_X|_{X^\prime})\to(Y^\prime,\E_Y|_{Y^\prime})$ is bornologous and effectively proper, since it is a restriction of $f$. Finally, since $f\colon X\to Y$ has uniformly bounded fibers, $X^\prime$ is large in $(X,\Sym(\E_X))$.

%(e)$\to$(d) Let $f^\prime\colon(X^\prime,\E_X|_{X^\prime})\to(Y^\prime,\E_Y|_{Y^\prime})$ be an asymorphism, where $X^\prime$ and $Y^\prime$ are large in $(X,\Sym(\E_X))$ and in $(Y,\Sym(\E_Y))$, respectively. Define $g\colon X\to X^\prime$ as follows. Let $E=E^{-1}\in\Sym(\E_X)\subseteq\E_X$ be an entourage such that $E[X^\prime]=X$ and, for every $x\in X$, let $z_x\in X^\prime$ such that $(x,z_x)\in E$. Then the map $g$ defined by the law $g(x)=z_x$, for every $x\in X$, is uniformly bounded preserving and uniformly bounded copreserving, since $(X,\E_X)$ is a quasi-coarse space, and it has uniformly bounded fibers, by construction. Hence $g$ is effectively proper and so $f=f^\prime\circ g$ is uniformly bounded preserving and effectively proper.

(e)$\to$(a) Let $M=M^{-1}\in\E_X$ be an entourage such that $M[X^\prime]=X$. Then define a map $h\colon X\to X^\prime$ as follows: if $x\in X^\prime$, then $h(x)=x$, and, if otherwise $x\in X\setminus X^\prime$, then $h(x)$ is a point such that $(h(x),x)\in M$. Similarly we can define a map $k\colon Y\to Y^\prime$. We claim that $h$ and $k$ are bornologous. Let $E\in\E_X$. Then note that $(h\times h)(E)\subseteq M\circ E\circ E\in\E_X$ and thus $h$ is bornologous. The same property can be similarly proved for $k$. Then the maps $f=f^\prime\circ h$ and $g=(f^{\prime})^{-1}\circ k$ are bornologous. We claim that $g$ is a $\Sym$-coarse inverse of $f$. For every $x\in X$, since $k|_{Y^\prime}=id_{Y^\prime}$,
$$
(x,g(f(x)))=(x,(f^\prime)^{-1}(k(f^\prime(h(x)))))=(x,(f^\prime)^{-1}(f^\prime(h(x))))=(x,h(x))\in M,
$$
and thus $g\circ f\sim_{\Sym}id_X$. The other request can be similarly proved.
\end{proof}

\begin{proposizione}
Let $(X,\E_X)$ and $(Y,\E_Y)$ be two $\Sym$-coarsely equivalent quasi-coarse spaces. If $(X,\E_X)$ is a coarse space, then so it is $(Y,\E_Y)$.
\end{proposizione}
\begin{proof}
Let $f\colon X\to Y$ be a $\Sym$-coarse equivalence and let $g\colon Y\to X$ be a $\Sym$-coarse inverse of $f$. Moreover, let $E=E^{-1}\in\E_X$ and $F^{-1}=F\in\E_Y$ be two symmetric entourages which witness that $g\circ f\sim_{\Sym} id_X$ and $f\circ g\sim_{\Sym} id_Y$, respectively. Then, for every $K\in\E_Y$ and $(x,y)\in K$,
$$
(y,x)=(y,f(g(y)))\circ(f(g(y)),f(g(x)))\circ(f(g(x)),x)\in F\circ(f\times f)(((g\times g)(K))^{-1})\circ F\in\E_Y,
$$
and then $K^{-1}\in\E_Y$.
\end{proof}

%Also in framework of quasi-balleans, boundedness and strong boundedness are different concepts, as the following two examples (Example \ref{ex:le} and \ref{ex:sorgenfrey}) show.

%\subsection{Quotient categories}

Let $\XX$ be a category and $\sim$ be a {\em congruence} on $\XX$, i.e., for every $X,Y\in\XX$, $\sim$ is an equivalence relation in $\Mor_\XX(X,Y)$ such that, for every $f,g\in\Mor_\XX(X,Y)$ and $h,k\in\Mor_\XX(Y,Z)$, $h\circ f\sim k\circ g$, whenever $f\sim g$ and $h\sim k$. Hence the {\em quotient category $\XX/_\sim$} can be defined as the one whose objects are the same of $\XX$ and whose morphisms are equivalence classes of morphisms of $\XX$, i.e., $\Mor_{\XX/_{\sim}}(X,Y)=\{[f]_\sim\mid f\in\Mor_{\XX}(X,Y)\}$, for every $X,Y\in\XX/_\sim$. For example, the closeness relation $\sim$ is a congruence in $\PCS$ and so the quotient category $\CS$ can be defined (\cite{DikZa}).

Let $\XX$ be one of the categories $\EN$, $\SPCS$ or $\QPCS$ and let $\Ffun\colon\XX\to\PCS$ be a functor. As it is shown in the first part of this section, an equivalence relation $\sim_{\Ffun}$ on $\XX$ can be induced, which is actually a congruence (it easily follows, since $\Ffun$ is a functor and $\sim$ is a congruence). Hence it is natural to produce the quotient category $\XX/_{\sim_{\Ffun}}$ and to compare it with $\CS$. In particular $\QPCS/_{\sim_{\Sym}}$ is worth being investigated.

\subsection{Characterisation of some special classes of quasi-coarse spaces}\label{sub:characterizing}

\begin{proposizione}\label{prop:connectedness}
Let $X$ and $Y$ be two $\Sym$-coarsely equivalent quasi-coarse spaces. Then:
\begin{compactenum}[(a)]
	\item $X$ is connected if and only if $Y$ is connected;
	\item $X$ is strongly connected if and only if $Y$ is strongly connected;
	\item $X$ is uniformly connected if and only if $Y$ is uniformly connected.
\end{compactenum}
%If $X$ satisfies a connectedness axiom, then the same axiom is satisfied by $Y$. Moreover, if $X$ satisfies (C$_1$) uniformly, then $Y$ satisfies (C$_1$) uniformly.%Connectedness, asymmetric connectedness, uniform weak connectedness, uniform weak directed connectedness, uniform weak asymmetric connectedness and weak asymmetric connectedness are properties of quasi-balleans which are invariant under $\Sym$-coarse equivalences.
\end{proposizione}
\begin{proof}
First of all, note that all those properties are invariant under asymorphism. Thanks to Theorem \ref{prop:sym-quasi-coarse-equivalence}(e), it is enough to prove the claim when $X$ is a large subspace of $\Sym(Y)$, which can be easily shown.%Moreover, let $Z^\prime$ be a subset of a quasi-coarse space $(Z,\E)$, which is large in $(Z,\Sym(\E))$. Then it is not hard to prove that $Z^\prime$ satisfies one of the connectedness axioms if and only if $Z$ does. Finally, Proposition \ref{prop:sym-quasi-coarse-equivalence}(e) implies the thesis.
\end{proof}

%Let $(X,\E)$ be an entourage structure. Define its {\em cofinality} as follows: $\cf\E=\inf\{\lvert\BB\rvert\mid\widehat{\BB}=\E\}$. If $(X,\E_d)$ is a metric entourage structure, then $\cf\E_d\leq\omega$. Moreover, for an entourage space $(X,\E)$ the following properties are equivalent:
%\begin{compactenum}[$\bullet$]
%	\item $\cf\E<\infty$;
%	\item $\cf\E=1$;
%	\item $\bigcup\E\in\E$;
%	\item $\E$ is a relation entourage structure (see Example \ref{ex:relation_entou}).
%\end{compactenum}

\begin{lemma}\label{lemma:cofinality}
Let $(X,\E_X)$ and $(Y,\E_Y)$ be two $\Sym$-coarsely equivalent quasi-coarse spaces. Then $\cf\E_X=\cf\E_Y$.
\end{lemma}
\begin{proof}
By applying Theorem \ref{prop:sym-quasi-coarse-equivalence}, we can assume that $Y$ is an entourage subspace of $X$ and the inclusion map $i\colon Y\to X$ is large-scale surjective. It is trivial that $\cf\E_Y\leq\cf\E_X$. Let $f\colon X\to Y$ be a $\Sym$-coarse inverse of $i$ and $M=M^{-1}\in\E_X$ be an entourage such that $(x,f(x))\in M$, for every $x\in X$. Then, for every base $\{E_i\}_{i\in I}$ of $\E_Y$, we claim that $\{M\circ E_i\circ M\}_i$ is a base of $\E_X$, and thus $\cf\E_X\leq\cf\E_Y$. In fact, let $F\in\E_X$ and $i\in I$ be an index such that $F|_{Y\times Y}\subseteq E_i$. Then $F\subseteq M\circ E_i\circ M$.
\end{proof}

We are now ready to prove the generalisations of some classical classification results in the framework of quasi-coarse spaces (\cite[Theorems 9.1, 9.2]{ProBan}, \cite[Theorem 2.11]{ProZar}). The following results, together with Proposition \ref{prop:metric}, give a complete characterisation of metric entourage structures.

\begin{teorema}\label{theo:ext_metriz}
Let $(X,\E)$ be a quasi-coarse space. The following properties are equivalent:
\begin{compactenum}[(a)]
\item there exists an extended quasi-metric $d\colon X\times X\to[0,\infty]$ on $X$ such that $\E=\E_{d}$; 
\item there exists an extended quasi-metric space $(Y,d)$ which is $\Sym$-coarsely equivalent to $(X,\E)$;
\item $\cf\E\leq\omega$.
\end{compactenum}
%Moreover, the quasi-metrics $d$ of items (a) or (b) can be taken as extended metrics if and only if $\E$ is a coarse structure.
\end{teorema}
\begin{proof}
The implications (a)$\to$(b)$\to$(c) are trivial: in particular, (b)$\to$(c) is implied by Lemma \ref{lemma:cofinality}.

(c)$\to$(a) Let $\{F_n\}_n$ be a base of $\E$ as in the proof of Proposition \ref{prop:metric}(a) with the following further property: for every $m,n\in\N$, $F_m\circ F_n\subseteq F_{m+n}$.  %For every pair of points $x,y\in X$, define
%\begin{equation}\label{eq:metric}
%d(x,y)=\begin{cases}\begin{aligned}&\min\{n\in\N\mid y\in E_n[x]\}&\text{if it exists,}\\
%&\infty &\text{otherwise.}\end{aligned}\end{cases}
%\end{equation}
We claim that the map $d\colon X\times X\to[0,\infty]$ defined as in \eqref{eq:metrizable} is an extended quasi-metric and, in order to show it, proving that $d$ satisfies the triangle inequality is enough. Let $x,y,z\in X$ be three arbitrary points. The inequality $d(x,z)\leq d(x,y)+d(y,z)$ trivially holds if $d(x,y)=\infty$ or $d(y,z)=\infty$. Suppose now that $d(x,y)\leq m$ and $d(y,z)\leq n$. Then $(x,z)=(x,y)\circ(y,z)\in F_m\circ F_n\subseteq F_{m+n}$ and thus $d(x,z)\leq m+n$. % and, if both of them are finite, the conclusion follows as in the proof of Theorem \ref{theo:metrizability}. Once again, it is easy to verify that $\E=\E_{d}$. 
Finally, the equality $\E=\E_d$ can be easily proved.
%The last claim follows from an easy variation of \cite[Proposition 2.1.1]{ProZar}.
\end{proof}
A quasi-coarse space satisfying the hypothesis of the previous theorem is called {\em quasi-metrizable}. Since the extended quasi-metric defined in the proof of Theorem \ref{theo:ext_metriz} is a quasi-metric if and only if the quasi-coarse space is strongly connected, in view of Proposition \ref{prop:connectedness}, Theorem \ref{theo:ext_metriz} can be specialised as follows.
\begin{corollary}\label{theo:metrizability}
	Let $(X,\E)$ be a quasi-coarse space. The following properties are equivalent:
	\begin{compactenum}[(a)]
		\item there exists a quasi-metric $d$ on $X$ such that $\E=\E_d$;
		\item there exists a quasi-metric space $(Y,d)$ which is $\Sym$-coarsely equivalent to $(X,\E)$;
		\item $(X,\E)$ is strongly connected and $\cf\E\leq\omega$.
	\end{compactenum}
	%Moreover, the quasi-metrics $d$ of items (a) and (b) can be taken as metrics if and only if $\BBB$ is a ballean.
\end{corollary}
\cite[Proposition 2.1.1]{ProZar} implies that the extended quasi metrics in Theorem \ref{theo:ext_metriz} and the quasi-metrics in Corollary \ref{theo:metrizability} can be taken as extended metrics and metrics, respectively, if and only if the initial space is a coarse space.

Finally we can answer to a problem posed by Protasov and Banakh \cite[Problem 9.4]{ProBan}.
\begin{teorema}\label{theo:graphic}
Let $(X,\E)$ be a connected quasi-coarse space. Then the following properties are equivalent:
\begin{compactenum}[(a)]
\item $(X,\E)$ is a graphic quasi-coarse space;
\item $(X,\E)$ is $\Sym$-coarsely equivalent to a graphic quasi-coarse space;
\item $(X,\E)$ is uniformly weakly connected and it is quasi-metrizable.
\end{compactenum}
\end{teorema}
\begin{proof}
The implications (a)$\to$(b) and (b)$\to$(c) are trivial.

(c)$\to$(a) Let $d$ be an extended quasi-metric such that $\E=\E_{d}$. Let $n$ be a natural number such that $\E_{d}$ is uniformly connected with parameter $E_n$. Finally, define a graph $\Gamma=(X,E)$, whose edges are the pairs $(x,y)\in X\times X$ such that $y\in E_n[x]$. Then the graphic quasi-ballean associated to the graph $\Gamma$ is asymorphic to $\E_{d}$ and so to $\E$.
\end{proof}

In the previous theorem, the request that the quasi-coarse space is uniformly connected can be easily relaxed. %If $(X,\E)$ is an arbitrary quasi-coarse space, one can consider its {\em (C$_1$)-components} (i.e., the equivalence classes under the equivalence relation $\leftrightsquigarrow$) and apply Theorem \ref{theo:graphic} to each of them. Then one can conclude that $(X,\E)$ is graphic if every weakly connected component satisfies (C$_1$) uniformly with the same parameter $E\in\E$.%there exists $\alpha\in P$ such that $\mathcal C_x$ is $\alpha$-weakly asymmetrically connected, for every $x\in X$. Hence the following Corollary can be stated.
\begin{corollario}
Let $(X,\E)$ be a quasi-coarse space. Then the following are equivalent:
\begin{compactenum}[(a)]
\item $(X,\E)$ is a graphic quasi-coarse space;
\item $(X,\E)$ is $\Sym$-coarsely equivalent to a graphic quasi-coarse space;
\item $(X,\E)$ is quasi-metrizable and there exists $E\in\E$ such that each connected component is uniformly connected with parameter $E$.
\end{compactenum}
\end{corollario}

%Let $S$ be a semigroup. Then $S$ is {\em left-amenable} if there exists a finitely additive left-invariant probability measure defined on $\mathcal P(S)$, i.e., a map $\mu\colon\mathcal P(S)\to[0,1]$ such that:
%\begin{compactenum}
%	\item[(a1)] $\mu(A\cup B)=\mu(A)+\mu(B)$, for every pair of disjoint subsets $A,B\in\mathcal P(S)$;
%	\item[(a2)] $\mu(S)=1$;
%	\item[(a3)] $\mu((s^\lambda_x)^{-1}(A))=\mu(A)$, for every $x\in M$ and $A\in\mathcal P(S)$, where $s^\lambda_x\colon S\to S$ is the left-shift by $x$.
%\end{compactenum}
%Similarly, $S$ is {\em right-amenable} if there exists $\mu\colon\mathcal P(S)\to[0,1]$ that satisfies (a1), (a2) and
%\begin{compactenum}
%	\item[(a3$^\prime$)] $\mu((s^\rho_x)^{-1}(A))=\mu(A)$, for every $x\in M$ and $A\in\mathcal P(S)$, where $s^\rho_x\colon S\to S$ is the right-shift by $x$.
%\end{compactenum}
%Finally, $S$ is {\em amenable} if it is both left and right-amenable. See \cite{CecCooKri} for a further details about amenable semigroups.

%\begin{question}
%Let $M$ and $N$ be two monoids and suppose that $f\colon(M,\E^\lambda_{[M]^{<\omega}})\to(N,\E^\lambda_{[N]^{<\omega}})$ is a $\Sym$-coarse equivalence. If it true that $M$ is left-amenable if and only if $N$ is left-amenable?
%\end{question}

\end{document}